\documentclass[11pt]{amsart}
\usepackage[dvips]{epsfig}
\usepackage{graphics}
\usepackage{latexsym}
\usepackage{verbatim}
\usepackage{amsmath}
\usepackage{amsthm}
\usepackage{amssymb}
\usepackage{hyperref}
\usepackage[]{hyperref}
\usepackage{float}
\usepackage{dsfont}

\hypersetup{
    colorlinks=true,%
    filecolor=black,%
    linkcolor=black,%
    urlcolor=black  %
}
\usepackage [table]{xcolor}
\usepackage{multirow}
\usepackage{float}
\usepackage{tikz}
\usepackage{subfig}
\usepackage{tikz-3dplot} 
\usepackage{textcomp}
\DeclareMathOperator*{\argmin}{arg\,min}

\graphicspath{{images/}} 
\newtheorem{theorem}{Theorem}[section]
\newtheorem{lemma}[theorem]{Lemma}

\newtheorem{proposition}[theorem]{Proposition}
\newtheorem{definition}[theorem]{Definition}
\newtheorem{remark}[theorem]{Remark}
\newtheorem{example}[theorem]{Example}

\newcommand\CC{\hbox{C\kern -.58em {\raise .54ex \hbox{$\scriptscriptstyle |$}}\kern-.55em {\raise .53ex \hbox{$\scriptscriptstyle |$}}}}
\newcommand\NN{\hbox{I\kern-.2em\hbox{N}}}
\newcommand\RR{\mathbb{R}}

\newcommand\ZZ{{{\rm Z}\kern-.28em{\rm Z}}}

\newcommand\ds{ \displaystyle }

\newcommand\bF{{\mathbf F}}

\newcommand\bU{{\mathbf U}}

\newcommand\ba{{\mathbf a}}
\newcommand\bb{{\mathbf b}}
\newcommand\bc{{\mathbf c}}

\newcommand\be{{\mathbf e}}

\newcommand\bk{{\mathbf k}}

\newcommand\br{{\mathbf R}}

\newcommand\bu{{\mathbf u}}
\newcommand\bv{{\mathbf v}}
\newcommand\bw{{\mathbf w}}
\newcommand\bx{{\mathbf x}}

\newcommand\bz{{\mathbf z}}
\newcommand\mA{{\mathcal A}}
\newcommand\mI{{\mathcal I}}
\newcommand\mK{{\mathcal K}}
\newcommand\mE{{\mathcal E}}
\newcommand\mC{{\mathcal C}}
\newcommand\mP{{\mathcal P}}

\def\signFF{\bigskip\bigskip\hspace{80mm}
\vbox{{\sc Francis Filbet\par\vspace{3mm}
Universit\'e de Toulouse III \& IUF \par
UMR5219, Institut de Math\'ematiques de Toulouse,\par
118, route de Narbonne\par
F-31062 Toulouse cedex,  FRANCE
\par\vspace{3mm}e-mail:} francis.filbet@math.univ-toulouse.fr }}

\def\signCP{\bigskip\bigskip\hspace{80mm}
\vbox{{\sc  C\'eline Parzani\par\vspace{3mm}
Ecole Nationale de l'Aviation Civile\par 
Laboratoire ENAC, \'equipe OPTIM\par
7 avenue Édouard-Belin, \par
BP 54005 Toulouse, FRANCE
\par\vspace{3mm}e-mail:}  celine.parzani@enac.fr }}

\begin{document}

\title[Vision based collision avoidance models]{On a three dimensional vision based collision avoidance model}

\author{C\'eline Parzani and Francis Filbet}

\maketitle

\begin{abstract}
This paper presents a three dimensional collision avoidance approach for aerial vehicles
inspired by coordinated behaviors in biological groups.
The proposed strategy aims to enable a group of vehicles to converge to a 
common destination point 
avoiding collisions with each other and with moving obstacles in their environment. 
The interaction rules lead the agents to adapt their velocity vectors
through a modification of 
the relative bearing angle and the relative elevation. Moreover the model satisfies the limited field 
of view constraints resulting from individual perception sensitivity.  
 
From the proposed individual based model, a mean-field kinetic model is derived.
Simulations are performed to show the effectiveness of the proposed model. 
 
\end{abstract}

\vspace{0.1cm}


\tableofcontents

\section{Introduction} 
\setcounter{equation}{0}
\label{sec:1}
In this paper we are interested in swarm modelling which represents the collective behavior 
of interacting agents of similar size and shape such that insects, birds or aerial vehicles. Inside the swarm, 
agents communicate with each other, working together to accomplish tasks and reach goals. 
As an example, in the last few years, the use of unmanned aerial vehicles swarm has been widely developed 
for numerous applications including monitoring of natural disasters, industrial accidents, surveillance of crowds, 
sensing in large environments, search and rescue missions, searching for sources of pollution, closed observation of 
protected areas and many others (see for instance \cite{Kopfstedt} or \cite{Han}). Main advantages are that 
the considered swarm can cover quickly a large area only requiring 
one operator or can scan high-risk sites rapidly whereas large vehicle cannot. 
All of these real-world challenges motivate serious investigations on how to control multiple vehicles cooperating automatically to accomplish a given task. 

On the other hand, nature provides great examples of decentralized, coordinated behaviors in groups 
of living organisms. Indeed, it is surprising how swarms of insects or flocks of birds
can travel in large, dense groups without colliding (see
\cite{Bonabeau, Camazine, Giardina} and \cite{Parrish}). Even in
the presence of external obstacles these agents are able to avoid collisions smoothly and such biological groups are 
remarkably effective at maintaining optimized group
structure, detecting and avoiding obstacles and predators, and
performing other complex tasks. Observing animals or pedestrians
collective motion, remarkable patterns are achieved by following simple rules. Such impressive inter-agent
coordination is accomplished despite their natural physiological constraints. Although individual agents 
have limited sensing capability and cannot see the whole formation, they can
form a flock with no apparent leader, which implies the lack of a centralized command. 
This highly coordinated collective behavior emerges from localized interactions among individuals within the swarm.

In this context, the objective of this paper is to propose a three dimensional model for a swarm of aerial vehicles 
inspired by coordinated behaviors of such biological groups. The following key points will be taken into account.
First, the model will be based on a sequence of simple rules followed by every individual (microscopic level). 
Then, it will include constraints related to limited sensor information. Moreover, since many applications occur 
in a high density traffic environment, the model will result in safe paths for all individuals.

To reach our objective, we consider an interacting particle system for the collective behavior
of swarms \cite{Carillo-1,Carillo-2}. In behavioral based methods, all the agents are considered equal and they
adopt behaviors built on informations coming from their only
neighborhood. Thanks to the feedback shared between
neighboring agents, these methods are following a decentralized
approach. In high
density traffic situations, it is recommended  to use a decentralized
coordination \cite{Hoekstra2000}, even if there is less freedom for maneuver. However, it is usually difficult to 
predict the group behavior, 
and the stability of the formation is generally not easy to prove either. These
methods are among the first to have been used in motion planning for multi-agent systems
as they are easily stated and generally efficiently scalable since their rules are supposed to be
implemented independently for each agent.

Safe paths is related to collision avoidance which plays an important role in the context
of managing multiple vehicles. It has been an active area of research in the field
of robotics using the collision cone method
\cite{ChakravarthyGhose1998} and the inevitable collision states
approach \cite{FraichardAsama2004, GomezFraichard2009}.  The collision
cone approach can be used to determine whether two objects, of
irregular shapes and arbitrary sizes, are on a collision course. It
has been the basis for many collision/obstacle avoidance algorithms \cite{ChakravarthyGhose1998}.
These methods are developed
with robotic application with  knowledge about the obstacles
(position, velocity, and acceleration)
\cite{GomezFraichard2009}. There have been also some research on
aircraft collision avoidance both from the multiple vehicles and the
air traffic control points of view. All these collision avoidance
procedures are based on three steps : see, detect, and avoid
\cite{Lacher2007}. But most of the algorithms developed
for air traffic management are those that guarantee safe trajectories
in a very low density traffic involving only two or three
aircraft. Another approach for collision avoidance is artificial
potential based methods where individuals are treated like charged
particles  of same charge that repel each other; whereas the
destination of an individual  is modeled as a charge of the opposite
sign so as to attract or navigate it toward the destination. The
artificial potential methods are susceptible to local minima and
require breaking forces \cite{Eby1994, EbyKelly1999}. 

In this paper, our goal is first to develop a three dimensional dynamical
approach describing the motions of $N$ individual and interacting
particles, when $N$ becomes large. The model is inspired from the ones developed in
\cite{Degond-1, Degond-2} and \cite{Tiwari, Tiwari2} for pedestrians 
collective motion in 2D but here we are concerned with 3D motion of aerial vehicles or birds which leads to an enhanced 
but more complex dynamics. 
Based on the vision based approach, we propose a model decomposed in two phases for collision avoidance 
including both particle to particle and moving obstacles avoidance. 

When dealing with large populations, in both cases one faces the well-known problem of the
curse of dimensionality, term first coined by Bellman precisely in the context of dynamic optimization: 
the complexity of numerical computations of the solutions of the above problems blows up as the size of
the population increases.  A possible way out is the so-called
mean-field approach, where the individual influence of the entire population on the dynamics of a single 
agent is replaced by an averaged one. This  substitution  principle  results  in  a  unique  mean-field 
equation  and  allows  the  computation  of solutions, cutting loose from the dimensionality.
Therefore, we perform a mean field limit of the microscopic model  to replace self-interactions 
between particles by self-consistent fields. The mean field approximation corresponds to the case where 
the force itself depends on some average of 
the distribution function. As a consequence, binary interactions between particles are not described 
but instead their global effect on each particle is taken into account. 
This approximation is justified especially in the configuration where the swarm is very closed to 
the target and therefore identifying binary interaction is very complex.  
As a result, we obtain a space-inhomogeneous kinetic PDE. 

The remainder of the paper is organized as
follows. In Section \ref{sec:2}, we present the individual agent based model proposed for 
self-propelled particle swarms including collision avoidance. In Section \ref{sec:3}, the associated 
mean-field limit is formally derived and analysed. 
Section \ref{sec:4} is devoted to numerical experiments of the microscopic model. We conclude with 
final remarks and future works in Section \ref{sec:5}. 

\section{Agent-based model for collision avoidance}
\setcounter{equation}{0}
\label{sec:2}

We are interested in modeling the motion of  individuals (vehicles, birds,..) with the objective to drive each 
individual of 
the swarm to a target point $\bx_T$ without colliding with any moving
obstacles or other individuals. 

Since we consider a swarm we do not explicitly constrain the relative location of each individual. 
This section is devoted to the presentation of the microscopic model  considering $N$ particles with position $\bx_i(t) \in \RR^3$ and
velocity $\bv_i(t) \in\RR^3$, with $1\leq i \leq N$. Then, we derive a
three-dimensional interacting particle system based on collision avoidance. The agent-based 
model we consider is inspired 
from the one proposed in \cite{Degond-1},\cite{Degond-2}  and \cite{Moussaid} developed for crowd dynamics. In these references, the heuristic-based model 
proposes that pedestrians follow a rule composed of two
phases:
\begin{enumerate} 
\item a perception phase; 
\item  a decision-making phase. 
\end{enumerate}

In the perception phase, the subjects make
an assessment of the dangerousness of the possible encounters in all the possible directions of motion. 
In the decision-making phase, they turn towards the direction which minimizes
the distance walked towards their target while avoiding encounters
with other pedestrians. Here, we mainly follow the same assumptions to
describe the perception phase, but then the individual changes its direction in order to diminish the 
probability of  collision. 

As we will describe later particles may accelerate or break smoothly according
to their distance to the target, but during the collision
avoidance process,  a sudden change of speed in the air is not
realistic, hence particles will only change their own
direction. Therefore, in the perception and decision making phases, we
assume that particles move with a constant speed, which means that
interacting particles cannot evaluate the change of speed of each
other. Of course in some situations, avoidance may fail when the
relative distance is too small or the relative velocity is too large
or when particles are not fast enough to change their direction. This
corresponds to physical situations where a crash cannot be
systematically avoided.  

\subsection{Perception Phase} We consider a particle $i\in\{1,\ldots,N\}$ located at a position $\bx_i(t)\in \RR^3$, with a 
velocity $\bv_i(t)$, interacting with a collision partner $j\in\{1,\ldots,N\}$ located at a position $\bx_j(t)\in \RR^3$, with a 
velocity $\bv_j(t)$. The sketch of the binary encounter between these
two particles is depicted in Figure \ref{fig:1}. In the sequel we
denote by $\langle.,.\rangle$ the usual scalar product in $\RR^3$,
$$
\langle \bu,\bv\rangle \,=\, \sum_{i=1}^3 u_i\,v_i
$$ 
and by  $|\bu|=\sqrt{\langle \bu\,,\,\bu\rangle}$ the associated norm.

We assume that $t=t^0$ is the time when particle $i$ evaluates the likeliness of a collision with particle $j$. 
This evaluation is made by supposing that each one maintains its velocity $\bv_i$, (respectively $\bv_j$) constant.
As depicted in Figure \ref{fig:1}, we introduce two notable points $\bar \bx_i$ and $\bar \bx_j$ that we define just below.

\begin{definition}
\label{def:int-points}
The interaction points $\bar \bx_i$ (resp. $\bar \bx_j$) of particle $i$ (resp. $j$) in their interaction
is the point $\bx_i(t)$ on the $i$-th particle's trajectory (resp. $\bx_j(t)$ on the $j$-th particle's
trajectory) such that $|\bx_i(t) - \bx_j(t)|$ is minimal, {\it i.e.}
$$
|\bar \bx_i - \bar \bx_j| = \min_{t\in \RR}|\bx_i(t) - \bx_j(t)|.
$$
\end{definition}

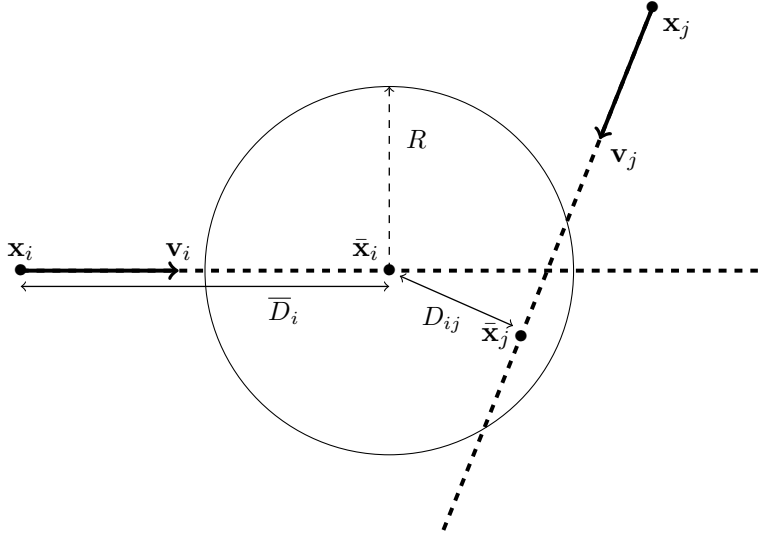
\begin{figure}[ht!]
\begin{center}
\begin{tikzpicture}
[scale=0.7]
\coordinate (xi) at (-10,0) ;
\coordinate (vi) at (-7,0) ;
\coordinate (xibar) at (-3,0) ;
\coordinate (xiend) at (4,0) ;
\coordinate (xj) at (2,5) ;
\coordinate (vj) at (1.,2.5) ;
\coordinate (xjbar) at (-0.5,-1.25) ;
\coordinate (xjend) at (-2,-5) ;
\coordinate (Rend) at (-3,3.5) ;
\draw [line width=1.5pt] [dashed] (xi)-- (xiend) ;
\draw [line width=1.5pt] [dashed] (xj)-- (xjend) ;
\draw [line width=1.5pt] [->] (xi)-- (vi) ;
\draw [line width=1.5pt] [->] (xj)-- (vj) ;
\coordinate (A) at (-10,-0.3) ;
\coordinate (B) at (-3,-0.3) ;
\draw [line width=0.5pt] [<->] (A)-- (B) ;
\coordinate (Dibar) at (-5.,-0.3) ;
\draw (Dibar) node[below ]{\small $\overline{D}_i$} ;
\coordinate (A) at (-2.8,-0.1) ;
\coordinate (B) at (-0.65,-1.05) ;
\draw [line width=0.5pt] [<->] (A)-- (B) ;
\draw (xi) node[above]{$\bx_i$} node{$\bullet$};
\draw (vi) node[above]{$\bv_i$} ;
\draw (xibar) node[above left]{$\bar \bx_{i}$} node{$\bullet$} ;
\draw (xj) node[below right ]{$\bx_j$} node{$\bullet$};
\draw (vj) node[below right]{$\bv_j$} ;
\draw (xjbar) node[left]{$\bar \bx_{j}$} node{$\bullet$};
\draw (xibar) circle (3.5) ;
\coordinate (Dij) at (-2.,-0.5) ;
\draw (Dij) node[below ]{\small $D_{ij}$} ;
\coordinate (R) at (-2.1,2.5) ;
\draw (R) node[left]{\small $R$} ;
\draw [line width=0.5pt][dashed] [->] (xibar)-- (Rend) ;
\end{tikzpicture}
\end{center}
\caption{\label{fig:1}Sketch of a binary encounter between two particles in 2D showing the key distances of the perception phase:   
the Minimal Distance $D_{ij}$ (distance between $\bar \bx_{i}$ and $\bar \bx_{j}$) and the Distance-To-Interaction $\overline{D}_i$ 
of particle $i$ in its interaction with particle $j$ (distance between the current particle position $\bx_i$ and $\bar \bx_i$). 
The circle with radius $R$ delimits the safety region for the particle $i$.
}
\end{figure}

\begin{definition}
\label{def:Dmin-TTI-DTI}
The interaction between particles $i$ and $j$ leads to define three key quantities associated to perception phase:
\begin{itemize}
\item The minimal distance $D_{ij}$ represents the smallest distance which separates the two particles $i$ and $j$ supposing 
that they cruise on a straight line at constant velocities $\bv_i$ and $\bv_j$. From Definition \ref{def:int-points}, the minimal distance 
is then the distance between the interaction points such that
$$
D_{ij} = |\bar \bx_i - \bar \bx_j|.
$$

\item The time-to-interaction $\tau_{ij}$ is the time needed by the subject to reach the
interaction point $\bar \bx_i$ from his current position $\bx_i =
\bx_i(t^0)$ at time $t^0$, which is counted positive if this
time belongs to the future of the subject and negative if it belongs to the past. Then, $\tau_{ij}$ is the value of $t$ for which the quantity $|\bx_i(t) - \bx_{j}(t)|$ is minimal. 

\item The distance-to-interaction $\overline{D}_i$ is the distance which separates the subject\textquotesingle s
current position $\bx_i = \bx_i(t^0)$ to the interaction point $\bar \bx_i$. The distance-to-interaction is counted positive
if the interaction point is reached in the future and negative if the
interaction point was crossed in the past:
$$
\overline{D}_i = {\rm sign}(t-t^0)\,|\bx_i - \bar \bx_i|,
$$
where $sign(t)$ denotes the sign of $t$. 
\end{itemize}
\end{definition}

\begin{remark}
Notice that the quantities $D_{ij}$ and $\tau_{ij}$ are symmetric with
respect to $i$ and $j$. Moreover, here, we have  supposed that each individual has a  perfect knowledge of
its own and partner’s positions and velocities, and we assume that
they are able to estimate or to compute the distance-to-interaction,
the minimal distance and the time to interaction with perfect accuracy
from the knowledge of $(\bx_i,\bv_i)$ and $(\bx_j,\bv_j)$.
\end{remark}

Let us now compute $\tau_{ij}$, $\overline{D}_i$ and $D_{ij}$ assuming
that a particle $i$ with a phase space position $(\bx_i,\bv_i)$ can detect
an interaction's partner $j$ located in its perception region with a
position $\bx_j$ and velocity $\bv_j$. We follow the same strategy as for two dimensional pedestrian
flow \cite{Degond-1} and  denoting by $\bx_i$ and $\bx_j$ the positions of
the two particles at time $t^0$, we define the distance 
$D(t)$ between the two particles at time $t \in (t^0, t^0+\delta t)$  by
\begin{equation}
D^2(t) = |\bx_j + \bv_j(t-t^0) - (\bx_i + \bv_i(t-t^0))|^2 
\label{Distance-2particles}
\end{equation}

Therefore, for each particle $i$ and its interaction partner $j$, we
have the following result.
 
\begin{proposition}
\label{prop:0}
The value of the time to interaction for the particle $i$, $\tau_{ij}$
is 
\begin{equation}
\label{TimeToInteraction}
\tau_{ij} = -\frac{\langle \bx_j-\bx_i \,,\, \bv_j- \bv_i\rangle}{|\bv_j-\bv_i|^2}, 
\end{equation}
whereas the distance to interaction $\overline{D}_i$ of particle $i$ and the minimal distance $D_{ij}$ are given by 
\begin{equation}
\label{Distance2a}
\displaystyle\overline{D}_i = -\frac{\langle\bx_j-\bx_i \,,\,\bv_j- \bv_i\rangle}{|\bv_j-\bv_i|^2}|\bv_i|
\end{equation}
and
\begin{equation}
\label{Distance2b}
 \displaystyle{D_{ij} = \left(|\bx_j - \bx_i|^2 - \left(\frac{\langle \bx_j-\bx_i \,,\,\bv_j-
  \bv_i\rangle}{|\bv_j-\bv_i|}\right)^2\right)^{1/2}}.
\end{equation} 
\end{proposition}

\begin{proof}
On the one hand, the value of the time to interaction for the particle
$i$, is  obtained by minimizing the quadratic function of time 
(\ref{Distance-2particles}) such that 
$$
D^2(t) = |\bv_j  - \bv_i |^2 \left( (t-t^0) + \frac{\langle\bx_j  - \bx_i\,,\,\bv_j  - \bv_i\rangle}{|\bv_j  - \bv_i |^2} \right)^2 + |\bx_j  - \bx_i |^2  - \frac{\langle\bx_j  - \bx_i\,,\,\bv_j  - \bv_i\rangle^2}{|\bv_j  - \bv_i |^2}, 
$$
hence it gives $\tau_{i,j}$ as in (\ref{TimeToInteraction}). Then,  the distance to interaction $\overline{D}_i$ of particle $i$ is  given by the distance traveled by this particle
during the time to interaction, that is,  $\overline{D}_i =
\tau_{ij}|\bv_i|$ where $\tau_{ij}$ is given by Definition
\ref{TimeToInteraction}. This leads to $\overline{D}_i$ as in
(\ref{Distance2a}). 

On the other hand, the minimal distance $D_{ij}$ is given by the
minimal value of (\ref{Distance-2particles}), it gives  $D_{ij} = D(t^0+\tau_{ij})$, 
which leads to (\ref{Distance2b}).
\end{proof}

The objective of the perception phase is to describe the configuration corresponding to a potential collision of the 
particle $i$ with the surrounding particles. From the definitions of
the minimal distance and time-to-interaction,  we consider that a collision may occur between particle $i$ and particle $j$ when the following conditions are satisfied. 
\begin{itemize}
\item First, we need  $\tau_{ij} > 0$ that means that we observe in the future.  
\item Second, if we define a safety zone for the particle $i$ delimited by the circle of radius $R$ as depicted in Figure 
\ref{fig:1} then collision will occur if $D_{ij} \le  R$. 
\end{itemize}
Combining these two conditions mean that in the future, the
trajectories of each particle will encounter inside the safety
zone. Therefore, we define the set of particles which may interact
with a particle $i$ located at $(\bx_i,\bv_i)\in\RR^3\times\RR^3$ at
time $t^0$, as 
$$
\mI_i(t^0) = \Big\{ j \in\{0,\dots,N\}, \,\tau_{ij} > 0,\,   D_{ij} \le  R \Big\}. 
$$
However, some restrictions related to the perception sensitivity of the individual (vision, sensors, etc) has also to be taken into account. 
As a consequence, considering a test particle $i$ interacting with
another particle $j\in \mI_i(t^0)$, we restrict the set of potential 
partner collision to those belonging to the ``vision cone'' of particle $i$ denoted $\mC_i$.  
This region is represented for instance as the blue area in Figure \ref{fig:2} and model the set of positions for the particle $j\in \mI_i(t^0)$ that are seen 
by the particle $i$. Let us now define the ``vision cone'' $\mC_i$ precisely. 

\begin{definition}[Vision cone]
\label{VisionCone}
Introducing a threshold number $\kappa \in [-1, 1]$, the ``vision cone'' $\mC_i$ for the particle $i$ is the cone centered at $\bx_i$ 
with angle $\cos^{-1}(\kappa)$ about the direction $\bv_i$.
\end{definition}
\begin{remark}
Observe that in Definition \ref{VisionCone}, we choose the vision
cone such that it has an
infinite radius, but the relative distance and velocity between two
particles will be taken into account thanks to the parameter
$\tau_{ij}>0$, where the collision avoidance's frequency will be a
decreasing function of $\tau_{ij}$. However, the present model can be
adapted without any difficulty to the case where the vision cone is
also limited by its distance.
\end{remark}

To summarize the perception phase, for each particle $i$ we define the
set of interaction's partners as the set 
\begin{equation}
\label{Ki}
\mK_i(t^0) \,=\,  \left\{ j\in \mI_i(t^0),\,\, \bx_j \in \mC_i\right\}.
\end{equation}

So we now detail the Decision Making Phase in order to model collision avoidance.

\subsection{Decision Making Phase}
\label{sec:2.2}
First let us emphasize  that the  three dimensional 
swarm modeling is quite different from the two dimensional case
encountered in collision avoidance for pedestrians or robots \cite{Moussaid}. Indeed,
in the three dimensional case, 
particles cannot suddenly stop or brake! 

Here we consider the motion of a particle $ i\in\{1,\ldots, N\}$ with position and
velocity $(\bx_i,\bv_i)\in\RR^3\times\RR^3$, which interacts with a particle
$j\in\{1,\ldots,N\}$ located at $(\bx_j,\bv_j)\in\RR^3\times\RR^3$. Depending on the position of the interaction points $(\bar \bx_i, \bar \bx_j)\in\RR^3\times\RR^3$, the collision avoidance procedure leads to consider three configurations:
\begin{itemize}
\item Safe configuration (illustrated in Figure \ref{fig:2}-(a)), where the particle $i$ does not change its
  direction and continues its cruise\,;
\item Blind configuration (illustrated in Figure \ref{fig:2}-(b)),
  where a collision is likely, but particle $i$ does not see $j$,
  hence  it continues its cruise whereas $j$ is expected to modify its
  direction\,;

\item Unsafe configuration (illustrated on Figure \ref{fig:2}-(c)), where the
  particle $i$ has detected an interaction's partner $j$ and both of them
  modify their direction.
\end{itemize}

\begin{figure}[ht!]
\begin{center}
\begin{tabular}{ccc}
\begin{tikzpicture}[scale=0.3]
\coordinate (O) at (0,0) ;
\coordinate (x) at (-4,-2) ;
\coordinate (y) at (8,0) ;
\coordinate (z) at (0,8);
\draw [line width=0.5pt] [->] (O)-- (x) ;
\draw [line width=0.5pt] [->] (O)-- (y) ;
\draw [line width=0.5pt] [->] (O)-- (z) ;
\draw (O) node[below]{$O$};
\draw (x) node[above]{$x$};
\draw (y) node[below]{$y$};
\draw (z) node[above]{$z$};
\coordinate (xi) at (2,3) ;
\draw (xi) node[below]{$\bx_i$} node{$\bullet$};
\coordinate (vi) at (6,3) ;
\draw [line width=1.5pt] [->] (xi)-- (vi) ;
\coordinate (xj) at (1,5) ;
\draw (xj) node[above]{$\bx_j$} node{$\bullet$};
\coordinate (vj) at (-1,2) ;
\draw [line width=1.5pt] [->] (xj)-- (vj) ;
\draw [fill=blue!80!black,opacity=0.05](xi) -- +(45:8cm) arc(45:-45:8cm) -- cycle ;
\end{tikzpicture} &
\begin{tikzpicture}[scale=0.3]
\coordinate (O) at (0,0) ;
\coordinate (x) at (-4,-2) ;
\coordinate (y) at (8,0) ;
\coordinate (z) at (0,8);
\draw [line width=0.5pt] [->] (O)-- (x) ;
\draw [line width=0.5pt] [->] (O)-- (y) ;
\draw [line width=0.5pt] [->] (O)-- (z) ;
\draw (O) node[below]{$O$};
\draw (x) node[above]{$x$};
\draw (y) node[below]{$y$};
\draw (z) node[above]{$z$};
\coordinate (xi) at (2,3) ;
\draw (xi) node[above]{$\bx_i$} node{$\bullet$};
\coordinate (vi) at (4,3) ;
\draw [line width=1.5pt] [->] (xi)-- (vi) ;
\coordinate (xj) at (-3,1) ;
\draw (xj) node[below right ]{$\bx_j$} node{$\bullet$};
\coordinate (vj) at (1,1) ;
\draw [line width=1.5pt] [->] (xj)-- (vj) ;
\draw [fill=blue!80!black,opacity=0.05](xi) -- +(45:8cm) arc(45:-45:8cm) -- cycle ;
\end{tikzpicture} & 
\begin{tikzpicture}[scale=0.3]
\coordinate (O) at (0,0) ;
\coordinate (x) at (-4,-2) ;
\coordinate (y) at (8,0) ;
\coordinate (z) at (0,8);
\draw [line width=0.5pt] [->] (O)-- (x) ;
\draw [line width=0.5pt] [->] (O)-- (y) ;
\draw [line width=0.5pt] [->] (O)-- (z) ;
\draw (O) node[below]{$O$};
\draw (x) node[above]{$x$};
\draw (y) node[below]{$y$};
\draw (z) node[above]{$z$};
\coordinate (xi) at (1,-1) ;
\draw (xi) node[below]{$\bx_i$} node{$\bullet$};
\coordinate (vi) at (4,2) ;
\draw [line width=1.5pt] [->] (xi)-- (vi) ;
\coordinate (xj) at (3,6) ;
\draw (xj) node[below right ]{$\bx_j$} node{$\bullet$};
\coordinate (vj) at (3,4) ;
\draw [line width=1.5pt] [->] (xj)-- (vj) ;
\draw [fill=blue!80!black,opacity=0.05](xi) -- +(90:8cm) arc(90:0:8cm) -- cycle ;
\end{tikzpicture} 
\\
(a)  & (b) & (c) 
\end{tabular}
\end{center}
\caption{\label{fig:2} Depending on the cone definition of the
  particle $i$, several configurations occur: (a) safe configuration where
  the two particles do not interact (b)
  blind configuration where $i$ does not interact with $j$, but $j$ is
  expected to change its direction (c) unsafe configuration where both
  particles will change their direction. }
\end{figure}
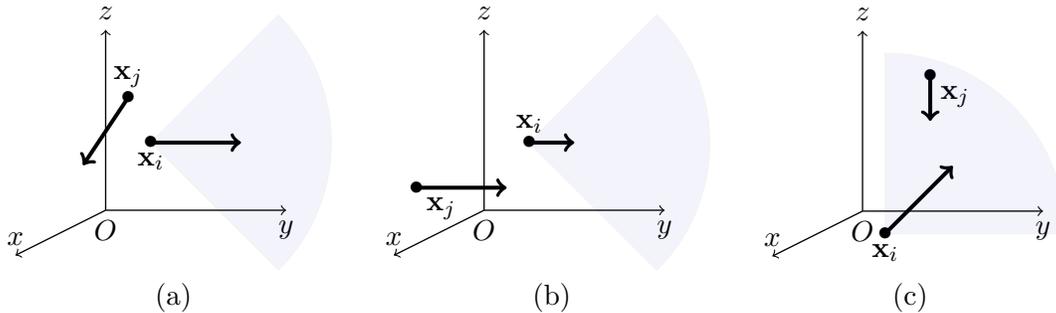

To describe more precisely this turning process,  we introduce the local frame of the particle
$i\in\{1,\ldots,N\}$  centered at position $\bx_i(t)\in\RR^3$, and  denoted by $(\be_{\rho_i}, \be_{\phi_i},
\be_{\theta_i})$  with  $\rho_i(t)=|\bv_i(t)|$, $\theta_i\in(0,2\pi)$ the azimuthal
angle and $\phi_i \in (0,\pi)$ the polar angle giving that $\bv_i=\rho_i\,\be_{\rho_i}$.

The collision avoidance model proposed below is based on the situation
where a particle $i\in\{1,\ldots,N\}$ interacts with another one $j\in
\mK_i(t^0)$ and  will modify its direction but preserve its speed,
that is, $\rho_i(t)$ is maintained constant during this process. To
determine this turning rate and the rotation axis, we need to
define some indicators on occurrence of collisions. The first indicator of the dangerousness of the collision is the time
$\tau_{ij}$, which indicates the remaining time before a collision
occurs.  The second  indicator   measured by
particle $i$, is the time derivative of the relative bearing angle or
azimuthal angle $\alpha_{ij}\in(0,2\pi)$ and the relative polar angle
$\beta_{ij}\in (0,\pi)$ formed  in its own frame between the direction  $\bv_i$ and the position $\bx_j$ of particle $j\in \mK(t^0)$ 
as depicted in Figure \ref{fig:3}.

To define rigorously these two angles and their time derivative we
need to consider the frame $(\be_{\rho_i}, \be_{\phi_i}, \be_{\theta_i})$ of
the particle $i$ at position $\bx_i\in\RR^3$ with velocity
$\bv_i\in\RR^3$.  

\begin{definition}[relative azimuthal and polar angles] 
Consider the local frame $(\be_{\rho_i}, \be_{\phi_i}, \be_{\theta_i})$ centered in at $\bx_i$ of
the particle $i\in\{1,\ldots,N\}$, and denote by $j\in\mK_i(t^0)$ its
collision partner located at $(\bx_j,\bv_j)\in\RR^6$. We define 
\begin{itemize}
\item the relative bearing or azimuthal angle $\alpha_{ij}\in (0,2\pi)$ as the azimuthal
angle of point $\bx_j$ with respect to the plane containing the point
$\bx_i$ and formed by the two vectors $(\be_{\rho_i}, \be_{\phi_i})$;

\item the relative polar angle $\beta_{ij}\in (0,\pi)$ as the
  polar angle of point $\bx_j$ with respect to the vector
  $\be_{\theta_i}$.

\end{itemize}
\label{def:bp}
\end{definition}

The choice of $\alpha_{ij}$ and $\beta_{ij}$ is here somehow arbitrary as long as
we obtain an orthonormal basis as we will see below.

\tdplotsetmaincoords{60}{110}
\pgfmathsetmacro{\rvec}{.9}
\pgfmathsetmacro{\thetavec}{30}
\pgfmathsetmacro{\phivec}{60}

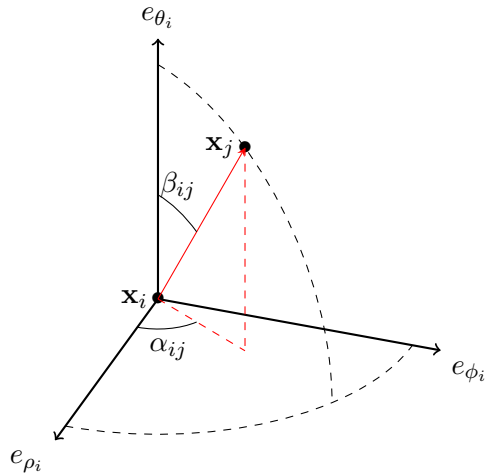
\begin{figure}[ht!]
\begin{center}
\begin{tikzpicture}[scale=4,tdplot_main_coords]
\coordinate (O) at (0,0,0);
\tdplotsetcoord{P}{\rvec}{\thetavec}{\phivec}
\draw[thick,->] (0,0,0) -- (1,0,0) node[anchor=north east]{$e_{\rho_i}$};
\draw[thick,->] (0,0,0) -- (0,1,0) node[anchor=north west]{$e_{\phi_i}$};
\draw[thick,->] (0,0,0) -- (0,0,1) node[anchor=south]{$e_{\theta_i}$};
\draw (O) node[left]{$\bx_{i}$} node{$\bullet$};
\draw (P) node[left]{$\bx_{j}$} node{$\bullet$};
\draw[-stealth,color=red] (O) -- (P);
\draw[dashed, color=red] (O) -- (Pxy);
\draw[dashed, color=red] (P) -- (Pxy);
\tdplotdrawarc{(O)}{0.2}{0}{\phivec}{anchor=north}{$\alpha_{ij}$}
\tdplotsetthetaplanecoords{\phivec}
\tdplotdrawarc[tdplot_rotated_coords]{(0,0,0)}{0.4}{0}{\thetavec}{anchor=south}{$\beta_{ij}$}
\draw[dashed,tdplot_rotated_coords] (\rvec,0,0) arc (0:90:\rvec);
\draw[dashed] (\rvec,0,0) arc (0:90:\rvec);
\end{tikzpicture}
\end{center}
\caption{\label{fig:3} Definition of the relative bearing angle $\alpha_{ij}\in (0,2\pi)$ as the azimuthal
angle of point $\bx_j$ in the frame $(\be_{\rho_i}, \be_{\phi_i}, \be_{\theta_i})$ centered at $\bx_i$ and of 
the relative polar angle $\beta_{ij}\in (0,\pi)$ as the polar angle of point $\bx_j$ in the frame 
$(\be_{\rho_i}, \be_{\phi_i}, \be_{\theta_i})$ centered at $\bx_i$.
}
\end{figure}

We also introduce the unit vector $\bk_{ij}$ of the line connecting the two
particles and the distance $d_{ij}$ between the particles. These quantities are defined by the following relations:
\begin{equation} 
\left\{
\begin{array}{ll}
d_{ij}(t) = |\bx_j(t) - \bx_i(t)| , 
\\ 
\,
\\
\ds\bk_{ij}(t) = \frac{\bx_j(t) - \bx_i(t)}{d_{ij}(t)}.
\end{array}\right.
\label{eq:defk}
\end{equation}

Then we perform a new change of frame with
respect to $\bx_i-\bx_j$, and introduce the orthonormal frame
defined as $(\bk_{ij}, \be_{\beta_{ij}}, \be_{\alpha_{ij}})$, where 
\begin{equation}
\left\{
\begin{array}{l}
\be_{\beta_{ij}} = \cos(\beta_{ij})\,\cos(\alpha_{ij})\,\be_{\rho_i}  \,+\,
  \cos(\beta_{ij})\,\sin(\alpha_{ij})\,\be_{\phi_i} \,-\,
  \sin(\beta_{ij})\,\be_{\theta_i}\,,
\\
\,
\\
\be_{\alpha_{ij}} = -\sin(\alpha_{ij})\,\be_{\rho_i} \,+\,
  \cos(\alpha_{ij})\,\be_{\phi_i}.
\end{array}\right.
\label{def:e}
\end{equation}
Notice that in three dimensions,  there are several possibilities to
define relative azimuthal and polar angles, but this choice is
arbitrary as long as we obtain an orthonormal basis. Then we  compute the time derivative of $\alpha_{ij}$ and
$\beta_{ij}$ which will be a key indicator in the  collision avoidance
process.

\begin{lemma}
Assume that particles $(i,j)$ are at time $t^0$ at positions $\bx_i$
and $\bx_j$, and move with constant velocity $\bv_i$ and
$\bv_j$. Then
\begin{equation}
\label{eq:bearing}  
\dot \beta_{ij} \,=\, \frac{1}{d_{ij}} \,\langle \bv_j - \bv_i ,
    \be_{\beta_{ij}} \rangle,\quad  \sin(\beta_{ij})\,\dot \alpha_{ij} \,=\, \frac{1}{d_{ij}}  \,\langle \bv_j - \bv_i ,   \be_{\alpha_{ij}} \rangle,
\end{equation}
\label{lmm:1}
\end{lemma}
\begin{proof}
By the definition of the relative bearing angle $\alpha_{ij}\in (0,2\pi)$ and
the relative polar angle $\beta_{ij}\in (0,\pi)$ , we can write:
$$ 
\bk_{ij} \,\,=\,\, \sin(\beta_{ij}) \,\cos(\alpha_{ij}) \,\be_{\rho_i} \,+\,
\sin(\beta_{ij}) \,\sin(\alpha_{ij}) \,\be_{\phi_i}   \,+\,
\cos(\beta_{ij}) \, \be_{\theta_i}. 
$$ 
Taking the time derivative of this relation and using the fact that
$(\be_{\rho_i}, \be_{\theta_i}, \be_{\phi_i})$  is constant since the motion of the
particle $i$ is supposed rectilinear with constant speed $\bv_i$, it  leads to 
\begin{eqnarray*}
\dot \bk_{ij} &=& \dot \beta_{ij} \,\left[
  \,\cos(\beta_{ij})\,\cos(\alpha_{ij})\,\be_{\rho_i} \,+\,
  \cos(\beta_{ij})\,\sin(\alpha_{ij})\,\be_{\phi_i} \,-\,
  \sin(\beta_{ij})\,\be_{\theta_i} \,\right]  
\\
&+& \,\sin(\beta_{ij})\,\dot\alpha_{ij} \, \,\left[
  \,-\sin(\alpha_{ij})\,\be_{\rho_i} \,+\,
  \cos(\alpha_{ij})\,\be_{\phi_i}\, \right],
\end{eqnarray*}
where we recognize the expression of the two unit vectors
$(\be_{\alpha_{ij}}, \be_{\beta_{ij}})$ constructed in (\ref{def:e})
by writing the point $\bx_j$ in spherical coordinates in the frame of
particle $i$. Hence we have 
\begin{equation}
\label{f:0}
\dot \bk_{ij}   = \dot \beta_{ij} \, \be_{\beta_{ij}} \,+\,  \sin(\beta_{ij})\,\dot\alpha_{ij} \, \be_{\alpha_{ij}}. 
\end{equation}
On the other hand, observing that 
$$
\dot{d_{ij}} \,=\,  \langle  \bv_{j}-\bv_i\,,\,
    \bk_{ij} \rangle
$$
and taking the time derivative of the first equation (\ref{eq:defk}),
it yields
\begin{eqnarray*}
\dot \bk_{ij}  &=&  \frac{d}{dt} \left( \frac{\bx_j - \bx_i}{d_{ij}} \right), 
\\
&=& \frac{1}{d_{ij}} \,\left[\, (\bv_{j}-\bv_i) \,-\,  \langle \bv_{j}-\bv_i\,,\,
    \bk_{ij} \rangle  \,  \bk_{ij}\,\right], 
\end{eqnarray*}
hence using the fact that $(\bk_{ij},   \be_{\alpha_{ij}}
\be_{\beta_{ij}} )$ constitutes an orthonormal basis, we finally have 
\begin{equation}
\label{f:1}
\dot \bk_{ij} \,=\,  \frac{1}{d_{ij}} \,\left[ \,\langle \bv_{j}-\bv_i ,
    \be_{\alpha_{ij}} \rangle  \,  \be_{\alpha_{ij}}  \,+\, \langle  \bv_{j}-\bv_i\,,\,
    \be_{\beta_{ij}} \rangle  \,  \be_{\beta_{ij}}\,\right]. 
\end{equation}
Identifying the two relations (\ref{f:0}) and (\ref{f:1}), we get $\dot \beta_{ij}$ and
$\sin(\beta_{ij})\,\dot \alpha_{ij}$, which gives rise to formula (\ref{eq:bearing}) for the derivative of
the relative bearing and polar angles. 
\end{proof}

\begin{figure}[ht!]
\begin{center}
\begin{tabular}{cc}
\begin{tikzpicture}[scale=0.5]
\coordinate (O) at (0,0) ;
\coordinate (x) at (-4,-2) ;
\coordinate (y) at (8,0) ;
\coordinate (z) at (0,8);
\draw [line width=0.5pt] [->] (O)-- (x) ;
\draw [line width=0.5pt] [->] (O)-- (y) ;
\draw [line width=0.5pt] [->] (O)-- (z) ;
\draw (O) node[left]{$O$};
\draw (x) node[above]{$x$};
\draw (y) node[below]{$y$};
\draw (z) node[above]{$z$};
\coordinate (xi) at (-1,-2) ;
\draw (xi) node[below]{$\bx_i$} node{$\bullet$};
\coordinate (vi) at (1,0) ;
\draw [line width=1.5pt] [->] (xi)-- (vi) ;
\coordinate (xj) at (4,7) ;
\draw (xj) node[below right ]{$\bx_j$} node{$\bullet$};
\coordinate (vj) at (4,5) ;
\draw [line width=1.5pt] [->] (xj)-- (vj) ;
\draw [fill=blue!80!black,opacity=0.05](xi) -- +(80:11cm) arc(80:10:11cm) -- cycle ;
\draw [fill=red!80!black,opacity=0.05](xj) -- +(-55:11cm) arc(-55:-125:11cm) -- cycle ;
\coordinate (xibar) at (3,2) ;
\coordinate (xiend) at (8,7) ;
\coordinate (xjbar) at (4,1.5) ;
\coordinate (xjend) at (4,-3) ;
\draw [line width=1.5pt] [dashed] (xi)-- (xiend) ;
\draw [line width=1.5pt] [dashed] (xj)-- (xjend) ;
\draw (xibar) node[above left]{$\bar \bx_{i}$} node{$\bullet$} ;
\draw (xjbar) node[right]{$\bar \bx_{j}$} node{$\bullet$};
\draw (xibar) circle (2) ;
\draw [line width=0.5pt][dashed] [->] (xibar)-- (xjbar) ;
\end{tikzpicture} & 
\begin{tikzpicture}[scale=0.5]
\coordinate (O) at (0,0) ;
\coordinate (x) at (-4,-2) ;
\coordinate (y) at (8,0) ;
\coordinate (z) at (0,8);
\draw [line width=0.5pt] [->] (O)-- (x) ;
\draw [line width=0.5pt] [->] (O)-- (y) ;
\draw [line width=0.5pt] [->] (O)-- (z) ;
\draw (O) node[left]{$O$};
\draw (x) node[above]{$x$};
\draw (y) node[below]{$y$};
\draw (z) node[above]{$z$};
\coordinate (xi) at (-1,-2) ;
\draw (xi) node[below]{$\bx_i$} node{$\bullet$};
\coordinate (vi) at (1,0) ;
\draw [line width=1.5pt] [->] (xi)-- (vi) ;
\coordinate (xj) at (2.5,-1) ;
\draw (xj) node[below right ]{$\bx_j$} node{$\bullet$};
\coordinate (vj) at (2.5,1.) ;
\draw [line width=1.5pt] [->] (xj)-- (vj) ;
\draw [fill=blue!80!black,opacity=0.05](xi) -- +(80:11cm) arc(80:10:11cm) -- cycle ;
\draw [fill=red!80!black,opacity=0.05](xj) -- +(55:11cm) arc(55:125:11cm) -- cycle ;
\coordinate (xibar) at (3.5,2.5) ;
\coordinate (xiend) at (8,7) ;
\coordinate (xjbar) at (2.5,3) ;
\coordinate (xjend) at (2.5,8) ;
\draw [line width=1.5pt] [dashed] (xi)-- (xiend) ;
\draw [line width=1.5pt] [dashed] (xj)-- (xjend) ;
\draw (xibar) node[right]{$\bar \bx_{i}$} node{$\bullet$} ;
\draw (xjbar) node[left]{$\bar \bx_{j}$} node{$\bullet$};
\draw (xibar) circle (2) ;
\draw [line width=0.5pt][dashed] [->] (xibar)-- (xjbar) ;
\end{tikzpicture}
\\
(a)  & (b)  
\end{tabular}
\end{center}
\caption{\label{fig:4} Vision cones for each particle ($\mC_i$ in blue and $\mC_j$ in pink) for the two considered configurations:  (a) cooperative interactions, (b) non cooperative interactions}
\end{figure}
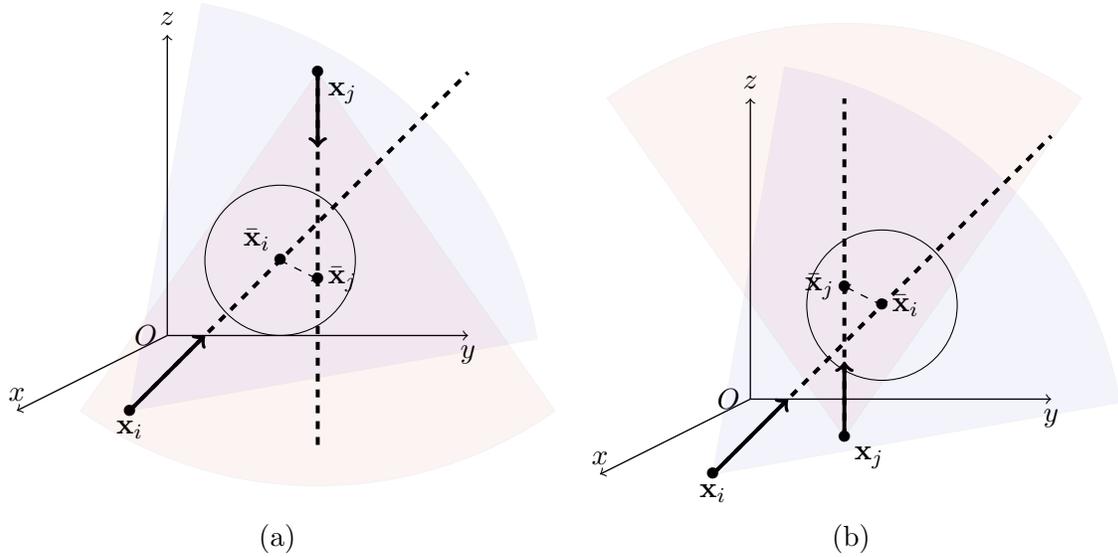

We are now ready to make the link between the time derivative of the relative
polar and bearing angles and the collision avoidance process.  Assume that $t=t^0$ and consider two particles $(i,j)\in
\{1,\ldots,,N\}^2$, such that $j\in\mI_i(t^0)$.  In the present situation the
two interaction points $\bar \bx_i$ and $\bar \bx_j$ are relatively
close  and the  particles need to rotate to avoid a collision, so that,
the minimal distance  $D_{ij}$, given in
(\ref{Distance2b}), will increase. Then, we write $\bv_j-\bv_i$ in the orthonormal frame
$\{\bk_{ij}, \be_{\beta_{ij}}, \be_{\alpha_{ij}}\}$ and using the
results of Lemma \ref{lmm:1}, it yields that 
$$
|\bv_j-\bv_i|^2 \,=\,  d_{ij}^2 \,\left(\,  |\dot\beta_{ij}|^2 \,+\, |\sin(\beta_{ij})\,\dot
  \alpha_{ij}|^2\,\right) \,+\, \langle \bv_j-\bv_i,\bk_{ij}\rangle^2,
$$
hence we have 
\begin{eqnarray*}
D_{ij}^2 &=& \left(\frac{d_{ij}}{|\bv_j-
  \bv_i|} \right)^2  \,\left(\, |\bv_j - \bv_i|^2 \,- \,\langle \bv_j-\bv_i,\bk_{ij}\rangle^2\,\right),
\\
 &=& \left(\frac{d_{ij}^2}{|\bv_j- \bv_i|} \right)^2 \,  \mA_{ij}^2, 
\end{eqnarray*}
where $\mA_{ij}^2$ is defined as
\begin{equation}
\mA_{ij}^2(t) \,\;:=\,\, \dot\beta_{ij}^2(t)\,+\,|\sin(\beta_{ij})\,\dot
  \alpha_{ij}(t)|^2.
\label{Aij}
\end{equation}
This results indicates that the
collision is very likely when $\mA_{ij}^2(t^0)$ is small. Therefore, to increase the minimal distance $D_{ij}$ we need to
increase the magnitude of the time derivative of the relative bearing and polar
angle $(\sin(\beta_{ij})\,\dot\alpha_{ij},\dot\beta_{ij})$ given in Lemma
\ref{lmm:1}. Thus, the proposed control scheme
is based on gyroscopic forces but adapted to the constraints due to the perception region. On the one
hand we consider the situation where the two particles see each other,
then they cooperate to avoid collisions (cooperative interaction
represented  in Figure \ref{fig:4}-(a)). On the other hand,
we describe the interaction of one particle with an obstacle or
another particle which do not deviate from its trajectory
(non-cooperative interaction represented in Figure \ref{fig:4}-(b)).

\subsubsection{Cooperative interactions}

At time $t=t^0$, both particles are such that
$(i,j)\in\mK_j(t^0)\times \mK_i(t^0)$ as it is shown in Figure \ref{fig:4}-(a). Then the two particles will
rotate in order to avoid to collide  along a rotation axis defined by
a  vector field  $\br_{ij}$ which has to be determined such that 
$$\left\{\begin{array}{l}
\ds\frac{d\bv_i}{dt} \,=\, \omega_{ij}\,\bv_i \wedge \br_{ij}, \\ \, \\ 
     \ds      \frac{d\bv_j}{dt} \,=\, \omega_{ij}\,\bv_j \wedge \br_{ij},
\end{array}\right.
$$ 
where $\omega_{ij}>0$ defines the rotation frequency. To this aim, we write the vector $\br_{ij}$ in the
basis $\{\bk_{ij}, \be_{\beta_{ij}}, \be_{\alpha_{ij}}\}$ as
\begin{equation}
\label{rij}
\br_{ij} \,\,=\,\, r^1_{ij}\,\bk_{ij} \,+\, r^2_{ij}\,\be_{\beta_{ij}}\,+\, r^3_{ij}\,\be_{\alpha_{ij}}
\end{equation}
and determine the values of $(r^1_{ij},r^2_{ij},r^3_{ij})$ in order to increase the
magnitude of $(\sin(\beta_{ij})\,\dot\alpha_{ij},\dot\beta_{ij})$. In
the next lemma, we determine the rotation
axis and show how to increase the time derivative of the bearing and
polar angles and therefore, thus decreasing the likeliness of the
collision. We follow the same strategy as \cite{Degond-2}
for two dimensional problems. 

\begin{lemma}
Assume that two particles $(i,j)\in\{1,\ldots,N\}^2$ are such that
$(i,j)\in\mK_j(t^0)\times \mK_i(t^0)$ and consider the time derivative of the relative bearing and polar angles
$(\sin(\beta_{ij})\,\dot\alpha_{ij}, \dot\beta_{ij})$ given in (\ref{eq:bearing}) and
the rotational axis $\br_{ij}$ given by (\ref{rij}) is such that $r^1_{ij}\in\RR$,
\begin{equation}
\label{toto2}
-\frac{\omega_{ij}\, r^3_{ij}}{\dot\beta_{ij}} \,\leq \,2\,\quad{\rm and}\quad \frac{\omega_{ij}\, r^2_{ij}}{\sin(\beta_{ij})\,\dot\alpha_{ij}} \,\leq\, 2.
\end{equation}
Then,  $(\sin(\beta_{ij})\,\dot\alpha_{ij}, \dot\beta_{ij})$  is solution to the following system
\begin{equation}
\label{confit2}
\left\{
\begin{array}{l}
\ds\frac{d\dot\beta_{ij}}{dt} \,\,=\,\,  \left( \omega_{ij} \,r^1_{ij} \,+\, \cos(\beta_{ij})\,\dot\alpha_{ij}\right)
  \,\sin(\beta_{ij})\,\dot\alpha_{ij} \,+\, \lambda^3_{ij}\,\dot\beta_{ij},
\\
\,
\\
\ds\frac{d}{dt}\left( \sin(\beta_{ij})\,\dot\alpha_{ij}\right) \,=\,
  -\left(\omega_{ij} \,r^1_{ij}\,+\, \cos(\beta_{ij})\, \dot\alpha_{ij}\right) \,\dot\beta_{ij}\, \,+\, \lambda^2_{ij}\,\sin(\beta_{ij})\,\dot\alpha_{ij},
\end{array}\right.
\end{equation}
with
\begin{equation}
\label{confit3}
\left\{
\begin{array}{l}
\ds\lambda^3_{ij} := \left(2 +
  \frac{\omega_{ij} \,r^3_{ij}}{\dot\beta_{ij}}\right)\,\left(\frac{|\bv_j -\bv_i|}{d_{ij}}\right)^2\,\tau_{ij} \in\RR^+,
\\
\,
\\
\ds\lambda^2_{ij} := \left(2-\frac{\omega_{ij} \,r^2_{ij}}{\sin(\beta_{ij})\,\dot\alpha_{ij}}\right)\,\left(\frac{|\bv_j
      -\bv_i|}{d_{ij}}\right)^2\,\tau_{ij} \in\RR^+.
\end{array}\right.
\end{equation}
Furthermore, $\mA_{ij}^2$ given in (\ref{Aij})  satisfies for $\gamma_{ij}=\min(\lambda^2_{ij},\lambda^3_{ij})$,
\begin{equation}
\label{res:Aij}
\frac{d\mA_{ij}^2}{dt} \,\geq \,\frac{\gamma_{ij}}{2}\, \mA_{ij}^2.
\end{equation}
\label{lmm:2}
\end{lemma}

\begin{proof} 
Let us consider the expression of $(\sin(\beta_{ij})\,\dot\alpha_{ij}, \dot\beta_{ij})$
given by (\ref{eq:bearing}). Then we compute the time derivative of
both quantities 
\begin{eqnarray*}
\frac{d\dot \beta_{ij}}{dt} &=&  \frac{1}{d_{ij}} \,\left( \langle \dot\bv_j - \dot\bv_i \,,\,
    \be_{\beta_{ij}} \rangle \,+\,  \langle \bv_j - \bv_i \,,\,
    \dot\be_{\beta_{ij}} \rangle \right)  
\\
&-& \frac{1}{d_{ij}^3}
  \langle \bv_j - \bv_i \,,\,  \bx_j - \bx_i \rangle \, \langle \bv_j - \bv_i \,,\,  \be_{\beta_{ij}} \rangle 
\end{eqnarray*}
and 
\begin{eqnarray*}
\frac{d}{dt}\left( \sin(\beta_{ij})\,\dot\alpha_{ij}\right) &=&  \frac{1}{d_{ij}} \,\left( \langle \dot\bv_j - \dot\bv_i ,
    \be_{\alpha_{ij}} \rangle \,+\,  \langle \bv_j - \bv_i ,
    \dot\be_{\alpha_{ij}} \rangle \right)  
\\
&-& \frac{1}{d_{ij}^3}  \langle \bv_j - \bv_i \,,\,  \bx_j - \bx_i \rangle\, \langle \bv_j - \bv_i \,,\,  \be_{\alpha_{ij}} \rangle . 
\end{eqnarray*}
Now we observe that 
$$
\left\{\begin{array}{l}
\dot\be_{\alpha_{ij}} \,=\, -\dot\alpha_{ij} \,\left[ \cos(\beta_{ij})  \,\be_{\beta_{ij}} \,+\, \sin(\beta_{ij})\,\bk_{ij} \right],  
\\
\,
\\
\dot\be_{\beta_{ij}} \,=\, +\dot\alpha_{ij} \, \cos(\beta_{ij}) \,  \be_{\alpha_{ij}} \,-\, \dot\beta_{ij}\,\bk_{ij},   
\end{array}\right.
$$   
hence using the definition of the unit vector $\bk_{ij}$ in (\ref{eq:defk}) and  the definition of $\tau_{ij}$ in (\ref{TimeToInteraction}), it
yields for the time derivative of the relative polar angle $\dot\beta_{ij}$,
$$
\frac{d\dot\beta_{ij}}{dt} \,=\,  \omega_{ij} \,\frac{\langle (\bv_j - \bv_i)\wedge \br_{ij} \,,\,\be_{\beta_{ij}} \rangle}{d_{ij}}   \,+\,  \cos(\beta_{ij}) \,\sin(\beta_{ij})\,\dot\alpha_{ij}^2  \,+\, 2\,\left(\frac{|\bv_j -\bv_i|}{d_{ij}}\right)^2\,\tau_{ij}\,\dot\beta_{ij},
$$
then for the time derivative of $\sin(\beta_{ij})\,\dot\alpha_{ij}$,
$$
\frac{d}{dt}\left( \sin(\beta_{ij})\,\dot\alpha_{ij}\right) \,=\,
\omega_{ij} \,\frac{\langle (\bv_j - \bv_i)\wedge \br_{ij} \,,\,    \be_{\alpha_{ij}} \rangle}{|\bx_j -\bx_i|} \,-\, \cos(\beta_{ij}) \,\dot\beta_{ij}\,\dot\alpha_{ij} \,+\, 2\,\left(\frac{|\bv_j
      -\bv_i|}{|\bx_j -\bx_i|}\right)^2\,\tau_{ij} \,\sin(\beta_{ij})\,\dot\alpha_{ij}.
$$
Therefore,  from the definition of $\br_{ij}$ in (\ref{rij}) and using
that $\langle \ba\,,\,\bb\wedge \bc \rangle = \langle
\bb\,,\,\bc\wedge \ba \rangle$, we get
$$
\left\{
\begin{array}{l}
\ds\langle (\bv_j - \bv_i)\wedge \br_{ij} \,,\, \be_{\beta_{ij}}
  \rangle \,=\, +r^1_{ij}  \,\langle \bv_j - \bv_i\,,\, \be_{\alpha_{ij}}\rangle
  \,-\, r^3_{ij} \,\langle \bv_j - \bv_i\,,\, \bk_{ij}\rangle,
\\
\,
\\
 \ds\langle (\bv_j - \bv_i)\wedge \br_{ij} \,,\, \be_{\alpha_{ij}}
  \rangle \,=\, -r^1_{ij}  \,\langle \bv_j - \bv_i\,,\, \be_{\beta_{ij}}\rangle
  \,+\, r^2_{ij} \, \langle \bv_j - \bv_i\,,\, \bk_{ij}\rangle,
\end{array}\right.
$$
which gives using (\ref{eq:bearing}), the  system of equations given in
(\ref{confit2}) with (\ref{confit3}).

From the assumption (\ref{toto2}), we get the nonnegativity of
the last coefficients. Therefore, multiplying the first equation of
(\ref{confit2}) by $\dot\beta_{ij}$ and the second one by
$\sin(\beta_{ij})\,\dot\alpha_{ij}$, it gives that
$$
2\,\frac{d\mA_{ij}^2}{dt} \,=\, \lambda^3_{ij} \,|\dot\beta_{ij}|^2 \,+\,
\lambda^2_{ij}\, \left|\sin(\beta_{ij})\,\dot\alpha_{ij}\right|^2\geq 0.
$$
Hence, from the nonnegativity of  $\lambda^2_{ij}$   and $\lambda^3_{ij}
$, the result (\ref{res:Aij}) follows. 
\end{proof}

\begin{remark}
Notice that  (\ref{res:Aij}) obtained in Lemma \ref{lmm:2} ensures
that the magnitude of $\mA_{ij}^2(t)$ will growth since $(\lambda_{ij}^2, \lambda_{ij}^3)$  given in
(\ref{confit3}) are  nonnegative. Furthermore, when $\lambda_{ij}^2$
and $\lambda_{ij}^3$  are bounded from below, $\mA_{ij}^2(t)$ will fast 
grow exponentially in
time.
\end{remark}

Applying Lemma \ref{lmm:2}, we observe that we can choose $\br_{ij}$
orthogonal to the unit vector $\bk_{ij}$ since this direction does not
have any effect on the variation  of $\mA_{ij}^2$.  We give a simple
choice for $\br_{ij}$.

\begin{example}
For any frequency $\omega_{ij}>0$, we take
$$
\br_{ij} \,:=\, -\frac{(\bv_j-\bv_i)\wedge \bk_{ij}}{d_{ij}}
$$
and after an easy computation, it gives 
$$
\left\{
\begin{array}{l}
\ds r^2_{ij}=\langle \br_{ij}\,,\, \be_{\beta_{ij}} \rangle \,=\, -\sin(\beta_{ij})\,\dot\alpha_{ij},
\\
\,
\\
\ds r^3_{ij}=\langle \br_{ij}\,,\, \be_{\alpha_{ij}} \rangle \,=\, \dot\beta_{ij},
\end{array}\right.
$$
hence, we have
$$
\left\{
\begin{array}{l}
\ds\frac{d\dot\beta_{ij}}{dt} \,=\,   \left( 2+\omega_{ij}\right)\,\left(\frac{|\bv_i -\bv_j|}{d_{ij}}\right)^2\,\tau_{ij}\,\dot\beta_{ij}\,+\,  \cos(\beta_{ij}) \,\sin(\beta_{ij})\,\dot\alpha_{ij}^2,
\\
\,
\\
\ds\frac{d}{dt}\left( \sin(\beta_{ij})\,\dot\alpha_{ij}\right) \,=\,
  \left( 2+ \omega_{ij}\right)\,\left(\frac{|\bv_i
      -\bv_j|}{d_{ij}}\right)^2\,\tau_{ij} \,\sin(\beta_{ij})\,\dot\alpha_{ij}\,-\, \cos(\beta_{ij}) \,\dot\beta_{ij}\,\dot\alpha_{ij}.
\end{array}\right.
$$
Then we have (\ref{res:Aij}) with 
$$
\gamma_{ij}\,:=\,\left( 2+\omega_{ij}\right)\,\left(\frac{|\bv_i -\bv_j|}{d_{ij}}\right)^2\,\tau_{ij}>0.
$$
\label{ex:1}
\end{example}
\subsubsection{Non-cooperative interactions}
\label{sec:nc-int}

Consider at $t=t^0$ two particles $(i,j)\in\{1,\ldots,N\}^2$ such that
$j\in\mK_i(t^0)$ but $i\notin \mK_j(t^0)$ as it is shown in Figure \ref{fig:4}-(b). Then only the particle $i$ will
rotate in order to avoid collision along a rotation axis defined by
a  vector field  $\br_{ij}$ which has to be determined such that 
$$\left\{\begin{array}{l}
\ds\frac{d\bv_i}{dt} \,=\, \tilde{\omega}_{ij}\,\bv_i \wedge \br_{ij}, \\ \, \\ 
     \ds      \frac{d\bv_j}{dt} \,=\, 0,
\end{array}\right.
$$ 
where $\tilde\omega_{ij} \in \RR$. Therefore we apply the same strategy as the one presented below to determine the condition for
which the time derivative of the polar angle
$\dot\beta_{ij}$ and $\sin(\beta_{ij})\,\dot\alpha_{ij}$ will
increase. Hence we prove the following result.

\begin{lemma}
Assume that two particles $(i,j)\in\{1,\ldots,N\}^2$ are such that
$j\in\mK_i(t^0)$ and $i\notin \mK_j(t^0)$ and consider the time derivative of the relative bearing and polar angles
$(\sin(\beta_{ij})\,\dot\alpha_{ij}, \dot\beta_{ij})$ given in (\ref{eq:bearing}) and
the rotational axis $\br_{ij}$ given by (\ref{rij}) is such that $r^1_{ij}=0$,
\begin{equation}
\label{toto3}
\tilde\omega_{ij}\,\frac{\cos(\alpha_{ij})\, r^3_{ij}}{\dot\beta_{ij}}\geq 0 \quad{\rm and}\quad
\tilde\omega_{ij}\,\frac{\cos(\alpha_{ij})\, r^2_{ij}}{\dot\alpha_{ij}} \leq 0.
\end{equation}
Then,  $(\sin(\beta_{ij})\,\dot\alpha_{ij}, \dot\beta_{ij})$  is solution to the following system
\begin{equation}
\label{c3}
\left\{
\begin{array}{l}
\ds\frac{d\dot\beta_{ij}}{dt} \,\,=\,\,  \cos(\beta_{ij})\,\sin(\beta_{ij})\,\dot\alpha_{ij}^2 \,+\, \eta^3_{ij}\,\dot\beta_{ij},
\\
\,
\\
\ds\frac{d}{dt}\left( \sin(\beta_{ij})\,\dot\alpha_{ij}\right) \,=\,
  - \cos(\beta_{ij})\, \dot\alpha_{ij}\,\dot\beta_{ij}\, \,+\, \eta^2_{ij}\,\sin(\beta_{ij})\,\dot\alpha_{ij},
\end{array}\right.
\end{equation}
with
$$
\left\{
\begin{array}{l}
\ds\eta^3_{ij} := \left(2\,\tau_{ij}\,\left(\frac{|\bv_j
      -\bv_i|}{d_{ij}}\right)^2\,+\,
  \tilde\omega_{ij}\,|\bv_i|\,\sin(\beta_{ij})\,\frac{\cos(\alpha_{ij}) \,r^3_{ij}}{\dot\beta_{ij}}\,\right)\,\in\,\RR^+,
\\
\,
\\
\ds\eta^2_{ij} := \left(2\,\tau_{ij}\,\left(\frac{|\bv_j
      -\bv_i|}{d_{ij}}\right)^2\,-\,
  \tilde\omega_{ij}\, |\bv_i|\,\frac{\cos(\alpha_{ij}) \, r^2_{ij}}{\dot\alpha_{ij}}\,\right)\,\in\,\RR^+.
\end{array}\right.
$$
Furthermore, $\mA_{ij}^2$ given in (\ref{Aij})  satisfies for $\gamma_{ij}=\min(\eta^2_{ij},\eta^3_{ij})$,
\begin{equation}
\label{res2:Aij}
\frac{d\mA_{ij}^2}{dt} \,\geq \,\frac{\gamma_{ij}}{2}\, \mA_{ij}^2.
\end{equation}
\label{lmm:3}
\end{lemma}

\begin{proof} 
We proceed as in the proof of Lemma \ref{lmm:2}, hence we get
$$
\frac{d\dot\beta_{ij}}{dt} \,=\,  -\tilde\omega_{ij}\,\frac{\langle \bv_i\wedge \br_{ij} \,,\,\be_{\beta_{ij}} \rangle}{d_{ij}}   \,+\,  \cos(\beta_{ij}) \,\sin(\beta_{ij})\,\dot\alpha_{ij}^2  \,+\, 2\,\left(\frac{|\bv_j -\bv_i|}{d_{ij}}\right)^2\,\tau_{ij}\,\dot\beta_{ij},
$$
and for the time derivative of $\sin(\beta_{ij})\,\dot\alpha_{ij}$,
$$
\frac{d}{dt}\left( \sin(\beta_{ij})\,\dot\alpha_{ij}\right) \,=\,
-\tilde\omega_{ij}\,\frac{\langle \bv_i\wedge \br_{ij} \,,\,    \be_{\alpha_{ij}} \rangle}{d_{ij}} \,-\, \cos(\beta_{ij}) \,\dot\beta_{ij}\,\dot\alpha_{ij} \,+\, 2\,\left(\frac{|\bv_j
      -\bv_i|}{d_{ij}}\right)^2\,\tau_{ij} \,\sin(\beta_{ij})\,\dot\alpha_{ij}.
$$
Furthermore, from the expression of $\be_{\beta_{ij}}$ and
$\be_{\alpha_{ij}}$ in (\ref{def:e}) and choosing $r^1_{ij}=0$, we get that
$$
\left\{
\begin{array}{l}
\ds\langle \bv_i\wedge \br_{ij} \,,\, \be_{\beta_{ij}}
  \rangle \,= \,-\, r^3_{ij} \,\langle \bv_i\,,\, \bk_{ij}\rangle
\;=\,-\, |\bv_i|\, \cos(\alpha_{ij}) \, \sin(\beta_{ij})\, r^3_{ij},
\\
\,
\\
 \ds\langle \bv_i\wedge \br_{ij} \,,\, \be_{\alpha_{ij}}
  \rangle \,=\, +\,r^2_{ij} \, \langle \bv_i\,,\, \bk_{ij}\rangle
\,=\,     +\,|\bv_i| \, \cos(\alpha_{ij}) \, \sin(\beta_{ij}) \, r^2_{ij}.
\end{array}\right.
$$
It gives the following system of equations 
 $$
\left\{
\begin{array}{l}
\ds\frac{d\dot\beta_{ij}}{dt} \,=\,  
                               \cos(\beta_{ij})\,\sin(\beta_{ij})\,\dot\alpha_{ij}^2 \,+\, \left(2 \,\tau_{ij}
  \,\left(\frac{|\bv_j -\bv_i|}{|\bx_j -\bx_i|}\right)^2 \,+\, \tilde\omega_{ij}\,|\bv_i|\,
  \sin(\beta_{ij})\,\frac{\cos(\alpha_{ij}) \,r^3_{ij}}{\dot\beta_{ij}}\right)\,\,\dot\beta_{ij},
\\
\,
\\
\ds\frac{d}{dt}\left( \sin(\beta_{ij})\,\dot\alpha_{ij}\right) \,=\,
                                                              
                                                              \,-\,
                                                              \cos(\beta_{ij})\,
                                                              \dot\alpha_{ij}\,\dot\beta_{ij}\,+\, \left(2\,\tau_{ij}\,\left(\frac{|\bv_j
      -\bv_i|}{|\bx_j -\bx_i|}\right)^2\,-\, \tilde\omega_{ij}\, |\bv_i|\,\frac{\cos(\alpha_{ij}) \, r^2_{ij}}{\dot\alpha_{ij}}\,\right) \, \sin(\beta_{ij})\dot\alpha_{ij}.
\end{array}\right.
$$
From the assumption (\ref{toto3}), we get the nonnegativity of
the last coefficients. Therefore, multiplying the first equation of
(\ref{c3}) by $\dot\beta_{ij}$ and the second one by
$\sin(\beta_{ij})\,\dot\alpha_{ij}$, it gives that
$$
2\,\frac{d \mA_{ij}^2}{dt}\,=\, \eta^3_{ij} \,|\dot\beta_{ij}|^2 \,+\,
\eta^2_{ij}\, \left|\sin(\beta_{ij})\,\dot\alpha_{ij}\right|^2\geq 0,
$$
hence (\ref{res2:Aij}) follows with $\gamma_{ij}=\min(\eta^2_{ij},\eta^3_{ij})$.
\end{proof}

Following Example \ref{ex:1}, we give a simple choice for $\br_{ij}$.
\begin{example}
For any frequency $\omega_{ij}>0$, we choose $\tilde\omega_{ij}=
\omega_{ij}\,\cos(\alpha_{ij})$ and  $\br_{ij}$ such that 
$$
\br_{ij} \,:=\, -\frac{(\bv_j-\bv_i)\wedge \bk_{ij}}{d_{ij}} 
$$
and as in  Example \ref{ex:1} we get
$$
\left\{
\begin{array}{l}
\ds r^2_{ij}=\langle \br_{ij}\,,\, \be_{\beta_{ij}} \rangle \,=\, - \sin(\beta_{ij})\,\dot\alpha_{ij},
\\
\,
\\
\ds r^3_{ij}=\langle \br_{ij}\,,\, \be_{\alpha_{ij}} \rangle \,=\, \dot\beta_{ij}.
\end{array}\right.
$$
Hence, from the choice of $\tilde\omega_{ij}$ and the latter
equalities, the assumption  (\ref{toto3}) is satisfied,  then we can apply Lemma \ref{lmm:3} and the particle
$i\in\{1,\ldots,N\}$ will deviate from $j\in\mK_i(t^0)$ whereas $j$
will continue its free motion. 
\label{ex:2}
\end{example}
\subsubsection{Collision avoidance model}

Finally, taking into account all the interactions between particles at
time $t=t^0$,
the force field applied for collision avoidance is given by the sum of
interactions as
\begin{equation}
\label{model:0}
\bF_i^{\rm self}(\bx_i, \bv_i)  \,\,=\,\, \frac{1}{N}
\sum_{j=1}^N  \omega_{ij}\, H_{ij}\,\,\mathds{1}_{\mK_i(t^0)}(j)\,\,  \bv_i \wedge\br_{ij}, 
\end{equation}
where $1_{\mK_i(t^0)}$ represents the characteristic function of the
set $\mK_i(t^0)$ defined in (\ref{Ki}),  $\omega_{ij}>0$ and  the
rotational axis $\br_{ij}$ is given by
\begin{equation}
\label{model:1}
\br_{ij} \,:=\, -\frac{(\bv_j-\bv_i)\wedge \bk_{ij}}{d_{ij}}\,,
\end{equation}
whereas the function $H_{ij}$ corresponds to either cooperative or
non-cooperative actions as explained above,
\begin{equation}
\label{model:1bis}
H_{ij} \,=\, \left\{
\begin{array}{ll}
1, & \textrm{if \,} i\in\mK_j(t^0),
\\
\cos(\alpha_{ij}), & \textrm{else.}
\end{array}\right.
\end{equation}

In the sequel the frequency $\omega_{ij}>0$ is chosen such that
$\omega_{ij}$ tends to zero when $\tau_{ij}\rightarrow +\infty$,
\begin{equation}
\label{model:2}
\omega_{ij} = \frac{8\,\pi}{|\br_{ij}|}\,e^{-\tau_{ij}}.
\end{equation}

\begin{remark}
Note that in some particular cases even when the set $\mK_i(t)$ is not
empty, the force term $\bF_i^{\rm self}(\bx_i, \bv_i)$ may be
zero. Indeed, it happens for instance 
\begin{itemize}
\item when $\bk_{ij}$ is colinear to
$\bv_j - \bv_i$, hence the vector $\br_{ij}=0$,
\item or when  the location of the set  particles in the vision cone of $i$ are
  perfectly symmetric with respect to the axis passing by $\bx_i$ of
  direction  $\bv_i$, hence $\bF_i^{\rm self}(\bx_i, \bv_i)$. 
\end{itemize}

 Therefore, in that case we choose it as
$$
\bF_i^{\rm self}(\bx_i, \bv_i) \,=\, \frac{\varepsilon}{N}
\sum_{j=1}^N  e^{-\tau_{ij}}\, H_{ij}\,\,\mathds{1}_{\mK_i(t^0)}(j)\,\,
\bv_i \wedge \be_z, 
$$ 
where $\varepsilon$ is chosen randomly and of order $10^{-6}$. In
practice this force term allows to break the symmetry and to remove
the degeneracy. 
\end{remark}

\subsection{Avoidance of obstacles and influence of the target}

Using the same strategy
as the one described below, obstacles $O\subset \RR^3$  are treated as particles,
where the particle interacts with the closest point belonging to the
intersection of the
obstacle and the vision cone of the particle $i$ at time $t^0$,
$$
\bx_O\, =\, \argmin_{\bx \in \partial O \cap \mK_i(t^0)} d\left(\bx_i(t^0),\bx\right),
$$ 
whereas $\bv_0\in\RR^3$ is the given velocity of the obstacle. 
Then the collision avoidance follows the same process as before except
that the obstacle does not deviate.
 
On the other hand, a force $-\nabla V(\bx_i) $ is applied to  steer particle $i$ to its
destination. The potential $V$ is the distance function 
$$
V(\bx_i) \,=\, | \bx_i - \bx_T |,
$$
where $\bx_T$ represents the location of the target, whereas a
friction term is added to control the speed of the particle $i\in\{1,\ldots,N\}$. Hence the particle  $i$ is directed by the sum of the
gradient of the potential field $-\nabla V(\bx_i)$ and the friction force in the following manner
$$
\bF_i^{\rm ext}(\bx_i, \bv_i)  \,\,=\,\,  -\nabla V(\bx_i) \,-\, \sigma \,\bv_i,
$$
where $\sigma>0$ represents the friction coefficient. This latter
force field induces a change of speed of particle $(\bx_i,\bv_i)$.
\subsection{Influence of the noise}
\label{subsec_noise}

Obviously, the motion of particles is not fully deterministic. When
some decisions need to be made in front of several alternatives, the
response of the subjects is subject-dependent. The simplest way to
model this inherent uncertainty consists in adding a Brownian motion
in velocity \cite{sde}
$$
d \bv_i \,=\,  \sqrt{2\,\nu} \circ dB^i_t, 
$$
where $\sqrt{2\,\nu}$ is the noise intensity and where $dB^i_t$ are
standard white noises in 3D, which are independent from one particle
to another one. The circle means that the stochastic differential
equation must be understood in the Stratonovich sense. The integration
of this stochastic differential equation generates a Brownian motion
\cite{Hsu_AMS02,sde}. This stochastic term adds up to the previous ones.

\subsection{Agent-based model for collision avoidance}
Finally from the requirements defined in the perception and decision
making phases, we get the following model constructed from the force
field $\bF_i^{\rm self}$ and $\bF_i^{\rm ext}$,
\begin{equation}
\left\{
\begin{aligned}
& \displaystyle{\frac{d \bx_i}{dt} \,=\, \bv_i,}\\
& \displaystyle{{d\bv_i} \,=\,  \left(\frac{1}{N}\,\sum_{j=1}^N \omega_{ij}\,H_{ij}\,\mathds{1}_{\mK_i(t)}(j)\,\,\bv_i\wedge \br_{ij} \,-\,\nabla V(\bx_i) \,-\,\sigma \,\bv_i \right) \,dt
  \,+\, \sqrt{2 \nu} \circ dB^i_t,}
\end{aligned}
\right.
\label{eq:dynamics}
\end{equation}
where $\br_{ij}$, $H_{ij}$ and $\omega_{ij}$ are given in (\ref{model:1})-(\ref{model:2}).

Note that in the two dimensional case, the interactions occur in the
horizontal plane and  the rotation axis is parallel to $Oz$, hence we recover the model proposed for
pedestrian in \cite{Degond-1, Degond-2} for binary
interactions. However, for multiple interactions the models differ
since in our approach,  the particle $i$ only rotates to avoid
collision among other particles without optimizing its trajectory to reach a
target as in \cite{Degond-1, Degond-2}. In (\ref{eq:dynamics}),
interacting particles are considered as obstacles where the intensity
of the force depends on the time to interaction $\tau_{ij}$ thanks to
the frequency $\omega_{ij}$ in (\ref{model:2}) : the probability to
deviate is small when $\tau_{ij}$ is high and is of order one when
$\tau_{ij}\rightarrow 0$.   This principle can be viewed as an
instantaneous reaction to avoid collision with particles around. 

\begin{proposition}
\label{prop:1}
Consider the solution $(\bx_i,\bv_i)_{1\leq i\leq N}$ to the
agent-based model (\ref{eq:dynamics}) without noise ($\nu=0$). Then
the energy given by
$$
\mE(t) \,:=\, \sum_{i=1}^N \left(\frac{|\bv_i |^2}{2} + V(\bx_i)\right),
$$
satisfies the following estimate 
$$
\frac{d\mE}{dt} \,\leq \, - \sum_{i=1}^N\sigma \,|\bv_i|^2.
$$
\end{proposition} 
\begin{proof}
Simply multiply the second equation of  (\ref{eq:dynamics}) by $\bv_i$
and integrate by part. By orthogonality property, we get the energy estimate.
\end{proof}
\section{Mean field kinetic model}
\label{sec:3}
We now consider the limit of a large number of particles $N
\rightarrow \infty$. We will give a formal proof of convergence when
there is no noise $\nu=0$ dealing with dynamical systems with
discontinuous coefficients and then with noise dealing with stochastic
differential systems.

\subsection{Mean field model without noise}
We first consider
the case without noise. For this derivation, we proceed like in \cite{spohn}. We
introduce the so-called empirical distribution $f^N (t, \bx,\bv)$
defined by
$$
f^N(t,\bx,\bv) \,:=\, \frac{1}{N} \,\sum_{i=1}^N \delta(\bx-\bx_i)\,  \delta(\bv-\bv_i),
$$
where $(\bx_i,\bv_i)_{1\leq i \leq N}$ is solution to the system of
ODEs (\ref{eq:dynamics}) with $\nu=0$. 

We introduce the cone $\mC(\bv)$ centered at the origin, with angle $\cos^{-1}(\kappa),\,\kappa\in[-1,1]$ about the direction $\bv\in\RR^3$
$$
\mC(\bv) \,:=\,\left\{\, \bz \in \RR^3, \quad \langle\bz,\bv\rangle\geq \kappa\, |\bz|\, |\bv|\,\right\}
$$
and for any $\bu\in\RR^3$, we set $\mI(\bu)$ as 
$$
\mI(\bu) \,:=\,\left\{\, \bz \in \RR^3, \quad \tau(\bz,\bu)>0, \,\,
  D(\bz,\bu) \leq R\,\right\},
$$
where the functions $D$ and $\tau$ correspond to 
\begin{equation}
\left\{
\begin{array}{l}
\displaystyle D (\bz,\bu) \,=\, \left(|\bz|^2 - \left(\bz .\frac{\bu}{|\bu|}\right)^2\right)^{1/2},
\\
\;
\\
\displaystyle \tau(\bz,\bu) \,=\, -\frac{\bz.\bu}{|\bu|^2}.
\end{array}\right.
\end{equation}
Finally, we define $\mK(\bv,\bw)\subset \RR^3$ as  
$$
\mK(\bv,\bw) \,=\,
\mI(\bw-\bv) \cap \mC(\bv).
$$ 

Then, it is an easy matter to
see that $f^N$ satisfies the following kinetic equation in the
distribution sense
\begin{equation}
\partial_t f^N + \bv\cdot \nabla_\bx f^N \,-\, \nabla_\bx V
  \cdot\nabla_\bv f^N  \,+\,  \nabla_\bv \cdot \left(\bv\wedge\Omega^N \,f^N\right) \,=\, \sigma\,\nabla_\bv \cdot
(\bv f^N), 
\label{fN}
\end{equation}
where $\Omega^N(t,\bx,\bv)$ is an interaction force defined by
\begin{equation}
\Omega^N(t,\bx,\bv) \,=\, -\frac{1}{N}\,\sum_{j=1}^N m(\bx_j-\bx,\bv,\bv_j) \,\mathds{1}_{\mK(\bv,\bv_j)}(\bx_j-\bx)  
  \,\,\br(\bx_j-\bx,\bv_j-\bv),
\label{OmegaN}
\end{equation}
with the rotation axis $\br(\bz,\bu)$ given by
$$
\br(\bz,\bu) \,=\, \frac{\bu\wedge \bz}{|\bz|^2},
$$
whereas the scalar function $m$ takes into account the frequency and the
cooperative and non-cooperative interaction,
$$
m(\bz,\bv,\bw) \,=\, \frac{8\pi}{|\br(\bz,\bw-\bv)|} \,
H_{\mC(\bw)}(\bz,\bv)\, \,e^{-\tau(\bx,\bw-\bv)} , 
$$
where 
$$
H_{\mC(\bw)}(\bz,\bv) \,=\, \left\{
\begin{array}{ll}
1, & \textrm{if \,} \bz\in\mC(\bw),
\\
\cos(\alpha(\bz,\bv)), & \textrm{else,\,} 
\end{array}\right.
$$
where $\alpha(\bz,\bv) \in (0,2\pi)$ corresponds to the relative
azimuthal angle of $\bz$ in the frame constructed from $\bv$ written in
spherical coordinates $\{\be_{|\bv|}, \be_{\phi}, \be_{\theta}\}$.
 
We note that relation  (\ref{OmegaN})  can be written
\begin{equation}
\Omega^N(t,\bx,\bv) \,=\, -\int_{\RR^3\times\RR^3} m(\bz,\bv,\bw) \,1_{\mK(\bv,\bw)}(\bz)  
  \,\,\br(\bz,\bw-\bv) \,f^N(t,\bx+\bz,\bw) d\bz\,d\bw,
 \label{OmegaNbis}
\end{equation}
which is a convolution product with respect to the  space variable
$\bx\in\RR^3$. Clearly, the formal mean-field limit of the particle system modeled by the kinetic
system (\ref{fN}), (\ref{OmegaNbis}) is given by the following system:
\begin{equation}
\left\{
 \begin{array}{l}
\ds\partial_t f + \bv\cdot \nabla_\bx f \,-\, \nabla_\bx V
  \cdot\nabla_\bv f  \,+\,  \nabla_\bv \cdot \left(\bv\wedge\Omega_f \,f\right) \,=\, \sigma\,\nabla_\bv \cdot
(\bv f), 
\\
\,
\\
\ds\Omega_f(\bx,\bv) = -\int_{\RR^3\times\RR^3} m(\bz,\bv,\bw) \,\mathds{1}_{\mK(\bv,\bw)}(\bz)  
  \,\,\br(\bz,\bw-\bv) \,f(t,\bx+\bz,\bw) d\bz\,d\bw,
\\
\,
\\
\ds f(t=0)\,=\, f_0 \in \mP_1 \cap L^\infty(\RR^6),
\end{array}\right.
\label{mf_trans}
\end{equation}
where we denote  by $\mP_1(\RR^6)$, the set of probability measures in $\RR^6$
with first bounded moment, which is a complete metric space endowed with the Monge-Kantorovich-Rubinstein distance. 
The Monge-Kantorovich-Rubinstein
distance, also called 1-Wasserstein distance, is also equivalent to the Bounded Lipschitz distance
$$
d_1(f, g) = \sup \left\{\left|\int_{\RR^6} \varphi(\bz) df(\bz) -
    \int_{\RR^6} \varphi(\bz) dg(\bz)\right|, \varphi\in {\rm
    Lip}(\RR^6), \,{\rm Lip}(\varphi) \leq 1  \right\},
$$
where Lip($\RR^6$) denotes the set of Lipschitz functions on $\RR^6$ and Lip($\varphi$) respectively the Lipschitz constant of a
function $\varphi$.

It is an open problem to rigorously show that this convergence
 holds. On the one hand, the lack of regularity of the
 velocity field in (\ref{eq:dynamics}) , due to the
sharpness of the sensitivity regions $\mK_i$ given in (\ref{Ki}),
prevents classical arguments from deriving rigorously the mean-field
limit. We refer to the recent work in \cite{Carillo-0}, where the
authors show the rigorous proof of the mean-field limit of a system of interacting
particles where each particle only interacts with those inside a local region whose shape depends on the
position and velocity of the particle. The argument is based on Filippov's theory \cite{filippov} allowing
to have a well-defined notion of solutions via differential inclusions. 
On the other hand, the additional work to take care concerns the
control of the error term in $d_1$ between weak solutions to
(\ref{mf_trans}) and empirical measures associated to differential
inclusions to (\ref{eq:dynamics}). 

Suppose that the empirical measure $f^N$ at time t = 0 converges in the weak star topology of bounded
measures towards a smooth function $f_0$ such that 
$$
d_1(f^N(0),f_0) \rightarrow 0, \quad{\rm when }\,N \rightarrow \infty.
$$
Following \cite{Carillo-0}, we may define the solution $f^N(t)$ to
(\ref{fN}) and $f(t)$ to (\ref{mf_trans})  thanks to the theory of
characteristics, hence it remains to establish a stability estimate as
$$
d_1(f^N(t),f(t)) \leq \, e^{C\,t} \, d_1(f^N(0),f_0), \quad \forall
\,t \,\in [0,T], 
$$
where $C>0$ is a positive constant depending on $f$. We will admit
that such a result is true and leave a rigorous convergence proof to
future work following \cite{Carillo-0}. 

For (\ref{mf_trans}) we can prove an analogous property as Proposition
\ref{prop:1} for (\ref{eq:dynamics})
\begin{theorem}
Consider a smooth potential $V(\bz)\geq 0$ and $V\in \mC^1(\RR^3)$. Assume that $f_0\in L^1 \cap L^\infty(\RR^6)$, with  $f_0\geq 0$ and 
$$
\int_{\RR^6} \left( |\bx|^2 + |\bv|^2\right) f_0(\bx,\bv) \,d\bx\,d\bv < \infty.
$$
Then for any $T>0$, there exists a  weak solution to (\ref{mf_trans}) such that
for almost every $t\in [0,T]$,
$$
f(t) \in L^1 \cap L^\infty(\RR^6),\quad \int_{\RR^6} \left( |\bx|^2 + |\bv|^2\right) f(t,\bx,\bv) \,d\bx\,d\bv < C(T,f_0)
$$
and for any $\varphi \in \mC_c^\infty([0, T)\times \RR^6)$,
\begin{equation}
\int_0^T\int_{\RR^6} f(t) \left( \partial_t \varphi +
    \bv\cdot\nabla_\bx\varphi - \left( \nabla_\bx V  - \bv\wedge
      \Omega_f + \sigma \bv \right) \cdot\nabla_\bv\varphi\right)\,
d\bx\,d\bv\,dt \,+\, \int_{\RR^6} f_0 \,\varphi(0) d\bx\,d\bv\,=\, 0.
\label{weak}
\end{equation}
Moreover, we have
$$
\frac{d}{dt} \int_{\RR^6} \left( \frac{|\bv|^2}{2} + V(\bx)\right)\,
f(t,\bx,\bv) d\bx\,d\bv  \,\leq\, -\sigma\,\int_{\RR^6} {|\bv|^2}\,
f(t,\bx,\bv) d\bx\,d\bv 
$$
\end{theorem}
\begin{proof}
We only give {\it a priori} estimates which allow to prove existence
of solutions by applying a classical regularizing process by
convolution. 

We consider a smooth solution to  (\ref{mf_trans}), which  is a six
dimensional advection equation in conservative form, hence we get
conservation of mass and nonnegativity of the solution leading to 
$$
\|f(t) \|_{L^1} = \|f_0\|_{L^1}, \quad \forall \,t \,\in \, [0,T].
$$
From this estimate and since $m(\bz,\bv,\bw)\,|\br(\bz,\bw-\bv)|\leq
1$, for any $\bz \in \mK(\bv,\bw)$, we prove that 
$$
\|\Omega_f \|_{L^\infty}\,\leq \, \|f(t)\|_{L^1} \,=\, \|f_0\|_{L^1}, \quad \forall \,t \,\in \, [0,T].
$$
Furthermore, since the advection field is locally bounded in $L
^\infty$, for any $p>1$, we multiply (\ref{mf_trans}) by
$p\,|f|^{p-1}$ and integrate over $(\bx,\bv)\in\RR^6$. It yields to the existence of 
a constant $C>0$ depending on $\|f_0\|_{L^1}$ such that
$$
\|f(t) \|_{L^p} = \|f_0\|_{L^p}\,e^{C\,t}, \quad \forall \,t \,\in \, [0,T].
$$ 
Next we multiply  (\ref{mf_trans}) by
$\frac{1}{2}|\bv|^2 + V(\bx)$ and integrate over
$(\bx,\bv)\in\RR^6$. After an integration by part, we get 
$$
\frac{d}{dt} \int_{\RR^6} \left( \frac{|\bv|^2}{2} + V(\bx)\right)\,
f(t,\bx,\bv) d\bx\,d\bv  \,\leq\, -\sigma\,\int_{\RR^6} {|\bv|^2}\,
f(t,\bx,\bv) d\bx\,d\bv,
$$ 
which allows to control the following quantity $|\bv|^2\, f(t) $ in
$L^1(\RR^6)$. Finally, we multiply  (\ref{mf_trans}) by
$|\bx|^2$ and integrate over $(\bx,\bv)\in\RR^6$,
\begin{eqnarray*}
\frac{d}{dt} \int_{\RR^6} {|\bx|^2}\, f(t,\bx,\bv) d\bx\,d\bv  &=& 2\,\int_{\RR^6} \bx\cdot \bv\, f(t,\bx,\bv) d\bx\,d\bv,
\\
&\leq & 2\,\left(\int_{\RR^6} |\bx|^2 f(t,\bx,\bv) d\bx\,d\bv\right)^{1/2}\; \left(\int_{\RR^6} |\bv|^2 f(t,\bx,\bv) d\bx\,d\bv\right)^{1/2},
\end{eqnarray*}
hence we get the estimate on the second order moment in space thanks
to the previous results.
 
Finally, from these {\it a priori} estimates, we get enough compactness to pass
to the limit in a regularized problem  in the nonlinear term
$f(t) \,\bv\wedge \Omega_f$ in (\ref{weak}) and prove existence of
weak solutions on any finite time interval $[0,T]$. 
 \end{proof}
Next, from the kinetic equation  (\ref{mf_trans}), we can construct an hydrodynamical system by considering a mono-kinetic
approximation given by 
$$
f(t,\bx,\bv) = \rho(t,\bx) \, \delta(\bv - \bU(t,\bx)),
$$
hence the couple $(\rho,\bU)$ is solution to 
$$
\left\{
\begin{array}{l}
\ds\partial_t \rho \,+\, \nabla_\bx\cdot (\rho\,\bU) \,=\, 0, 
\\
\,
\\
 \ds\partial_t \rho\bU \,+\, \nabla_\bx\cdot (\rho\,\bU\,\otimes\,\bU)
  \,\,=\,\, -\nabla_\bx V \,\rho  \,+\, \rho\,\bU \wedge \Omega \,-\, \sigma \,\rho\, \bU, 
\end{array}\right.
$$
where $\Omega$ is given by
$$
\Omega(t,\bx)\,=\, -\int_{\RR^3}
m(\bz,\bU(t,\bx),\bU(t,\bx+\bz))\,
\mathds{1}_{\mK(\bU(t,\bx),\bU(t,\bx+\bz))}(\bz)\,\br(\bz,\bU(t,\bx+\bz),
\bU(t,\bx) ) \,\rho(\bx+\bz) \,d\bz.
$$

\subsection{Mean field model with noise}
We now consider the case (\ref{eq:dynamics}) with Gaussian noise
$\nu>0$. This problem has been adressed in \cite{bolley}, where the
authors show that the $N$ interacting processes
$(\bx_{i,t},\bv_{i,t}))_{t\geq 0}$ respectively behave as
$N\rightarrow \infty$ like the auxiliary processes
$(\bar\bx_{i,t},\bar\bv_{i,t})_{t\geq 0}$, solutions to
\begin{equation}
\label{ajac}
\left\{
\begin{array}{l}
d \bx_{i,t} = \bv_{i,t} \,dt,
\\
d \bv_{i,t} = \, \left(\bv_{i,t} \,\wedge \Omega_{f_t}(\bx_{i,t}, \bv_{i,t}) \,-\,\nabla V(\bx_{i,t}) \,-\,\sigma \,\bv_{i,t} \right) \,dt
  \,+\, \sqrt{2 \nu} \circ dB^i_t,
\end{array}\right.
\end{equation}
where $f_t:={\rm law}(\bx_{i,t},\bv_{i,t})$ and 
$$
\Omega_{f_t}(\bx_{i,t}, \bv_{i,t})  \,=\, -\int_{\RR^3\times\RR^3} m(\bz,\bv_{i,t},\bw) \,\mathds{1}_{\mK(\bv,\bw)}(\bz)  
  \,\,\br(\bz,\bw-\bv_{i,t}) \,f(t,\bx+\bz,\bw) d\bz\,d\bw. 
$$
Note that (\ref{ajac}) consists of $N$ equations which can be solved independently of each other. Each
of them involves the condition that $f_t$ is the distribution of $(\bx_{i,t}, \bv_{i,t})$, thus making it nonlinear.
The processes $(\bx_{i,t}, \bv_{i,t})_{t\geq 0}$ with $i
\in\{1,\ldots,N\}$ are independent since the initial conditions and
driving Brownian motions are independent.

We will  admit that these processes defined on $\RR^6$ are identically
distributed, and their common law $f_t$ at time $t$, as a measure on
$\RR^3\times \RR^3$ evolves according to the following Kolmogorov-Fokker-Planck equation
\begin{equation}
 \partial_t f + \bv\cdot \nabla_\bx f \,-\, \nabla_\bx V
  \cdot\nabla_\bv f  \,+\,  \nabla_\bv \cdot \left(\bv\wedge\Omega_f \,f\right) \,=\, \nabla_\bv \cdot
(\nu\,\nabla_\bv f \,+\,\sigma\bv f).
\label{mf_f}
\end{equation}

\section{Numerical experiments}
\setcounter{equation}{0}
\label{sec:4}
In this section we present simulations to show the effectiveness of
the collision avoidance procedure proposed in this paper at the
microscopic level (\ref{eq:dynamics}) with
(\ref{model:1})-(\ref{model:2}). 

We choose a smooth external potential $V$ such that 
$$
V(\bx) = \frac{1}{4}\,\left(1+|\bx-\bx_T|^2\right)^{1/2} 
$$
and the friction coefficient is fixed to $\sigma=1/4$. Furthermore to
emphasize the effect of the collision avoidance process we neglect the
noise and set $\nu=0$ in our simulations.

\subsection{Collision avoidance in the horizontal plane}

We first consider the simple situation where all particles move in a
direction parallel to the horizontal plane. Initially, all the particles are 
located in a circle and want to move on the opposite direction with an
initial velocity $\bv(0)= -\bx(0)/2$. Therefore in this very specific
situation, the collision point of all particles is the center of the circle. 

We consider the microscopic model \eqref{eq:dynamics} without any
noise $\nu=0$ and choose the radius of the circle delimiting the safety region  as depicted in Figure \ref{fig:1} such that $R=1$. For the vision cone given in
Definition \ref{VisionCone} we take $\kappa=\cos(2\pi/3)$ whereas the axis of rotation and the  turning
frequency are given in  (\ref{model:0})-(\ref{model:2}). Since the motion occurs in the horizontal plane, we
expect  the axis of rotation $r_{ij}$ to be colinear to the unit
vector $\be_z$. In Figure \ref{fig:test0-1}, we present the numerical results with two, three, four and nine particles and
observe that the present model preserves perfectly the symmetry.
Furthermore, due to the perception phase, the collision is anticipated which
seems to guarantee a smooth trajectory and not a brutal change of direction.

\begin{figure}[ht!]
\begin{center}
 \begin{tabular}{cc}
\includegraphics[width=8.cm]{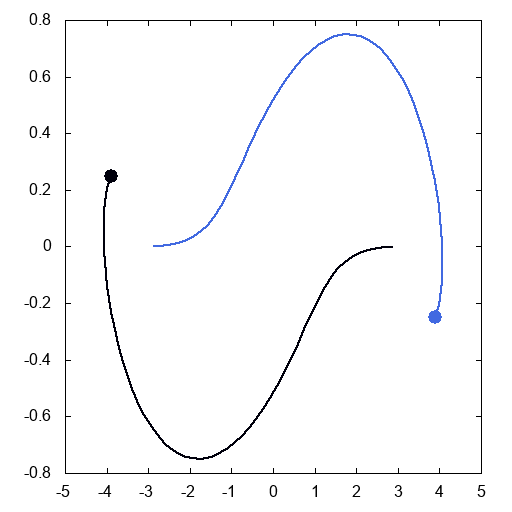} &    
\includegraphics[width=8.cm]{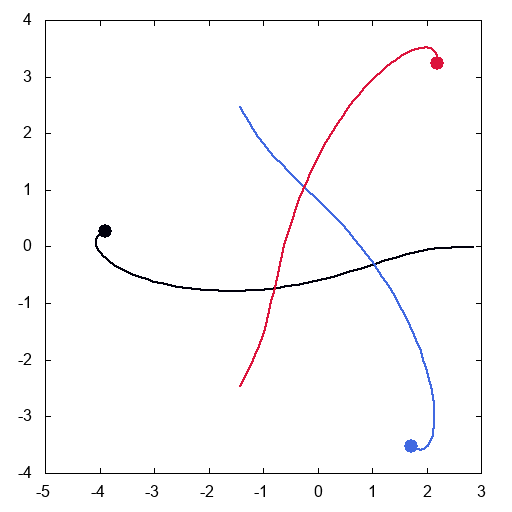} 
\\
(a)  & (b)  
\\
\includegraphics[width=8.cm]{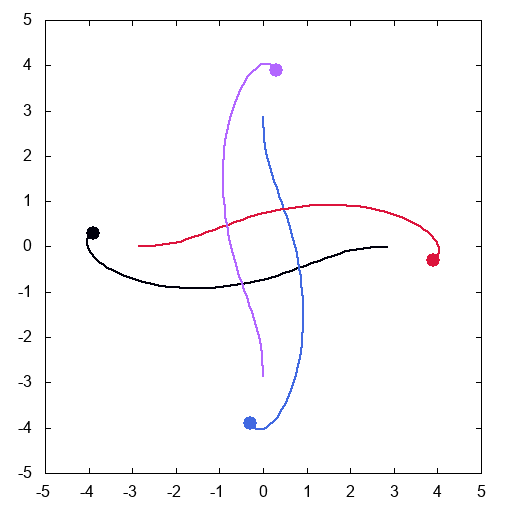} &    
\includegraphics[width=8.cm]{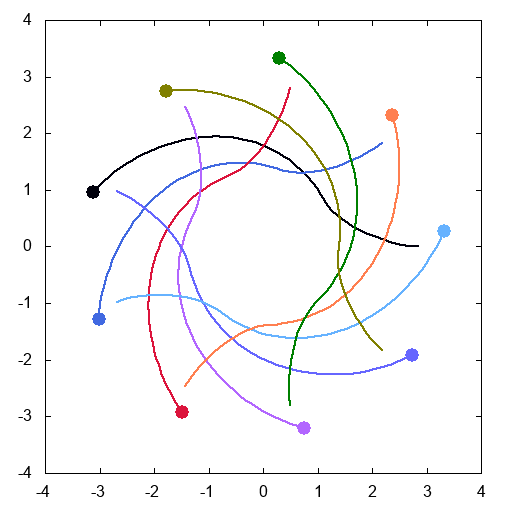} 
\\
(c)  & (d)  
\end{tabular}
\caption{\label{fig:test0-1}
{\bf Collision avoidance in the horizontal plan.} space trajectory in
the horizontal plane for (a) 2 particles, (b) 3 particles, (c) 4
particles and  (d) 9 particles.}
 \end{center}
\end{figure}

These numerical results reproduce the classical trajectories as in
\cite{Roelofsen}. The particles move in a straight line to its own
target, then when it approaches the collision point, it starts to
rotate and finally deviates again to reach the target point.
\\

Furthermore, we illustrate the fact that in some situations, collision
cannot be avoided in particular when  the relative velocity between
interacting particles is too large, that is, the distance between two
interacting particles maybe very small.  For instance, we consider the previous situation
with three particles localized on the unit circle  but now we vary the modulus of the initial
velocity $\bv(0)=-\alpha\,\bx(0)$, with respect to $\alpha>0$.  The
results are presented in Figure \ref{fig:test0-2}, as it
is expected when the velocity of particles is too large, the distance
between  particles becomes smaller and smaller.

\begin{figure}[ht!]
\begin{center}
 \begin{tabular}{cc}
\includegraphics[width=8.cm]{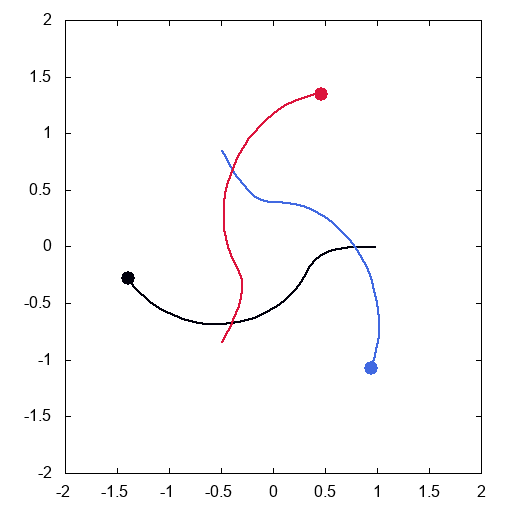} &    
\includegraphics[width=8.cm]{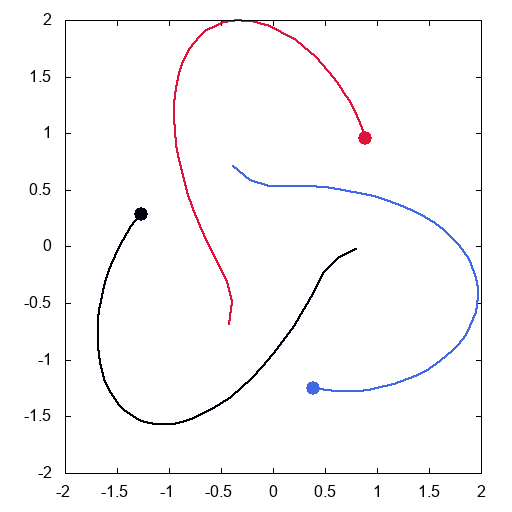} 
\\
(a)  & (b)  
\\
\includegraphics[width=8.cm]{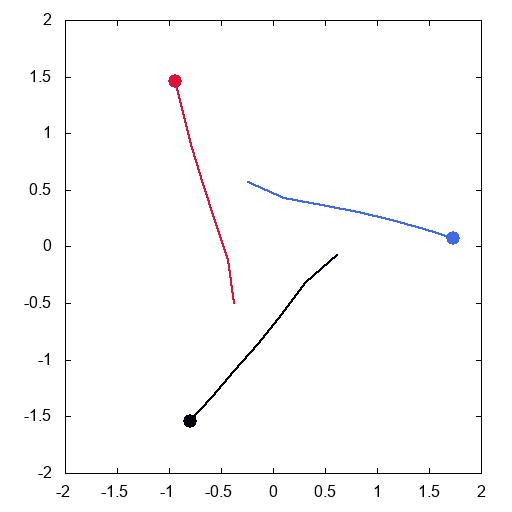} &    
\includegraphics[width=8.cm]{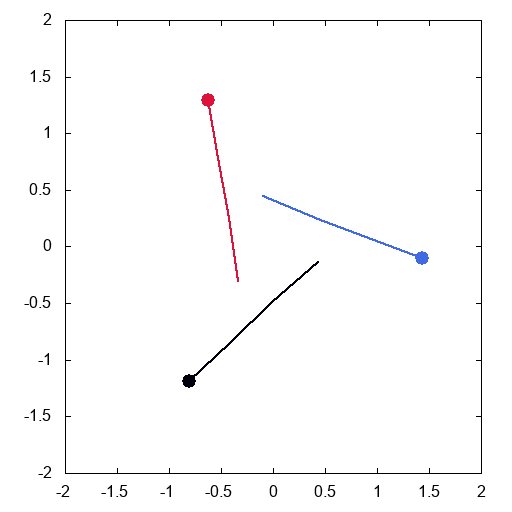} 
\\
(c)  & (d)  
\end{tabular}
\caption{\label{fig:test0-2}
{\bf Collision avoidance in the horizontal plan.} space trajectory in
the horizontal plane for three particles with different initial
velocities (a) $\bv(0)=-\bx(0)/10$, (b) $\bv(0)=-\bx(0)$,  (c) $\bv(0)=-2\bx(0)$ and  (d) $\bv(0)=-3\bx(0)$.}
 \end{center}
\end{figure}

\subsection{Influence of the vision cone}
We still consider the motion in the horizontal plane, but now the
particles are almost aligned to the $Ox$ axis and move initially along
this line where the particle behind has a larger speed than the one in
front of it, that is, for a small parameter $\epsilon=10^{-6}$, we
choose $\bx_1(0)=(-4,\epsilon,0)$ and $\bv_1(0)=(1,0,0)$, whereas $\bx_2(0)=(-2,0,0)$ and $\bv_2(0)=(1/2,0,0)$.

Furthermore, for each particle the target is also on the same
line. Thus, it is expected that the particle $(\bx_1,\bv_1)$ turns in order to
avoid a collision with $(\bx_2,\bv_2)$ whereas due to the restriction
of the vision cone, the second particle does not see the first one,
hence it continues its cruise in a straight line. 

Finally we also consider the same situation with three particles with
$\bx_3(0)=(-6,2\,\epsilon,0)$ and $\bv_3(0)=(2,0,0)$.
 
We present the numerical experiment in Figure \ref{fig:test1-1} for
two and three particles. In the first situation, we observe that
indeed the first particle deviates in order to avoid the collision,
whereas in the presence of three particles, the first one deviates
much more in order to avoid the collision with the second and the
third ones. The particle located in the front does not see the other
one coming from behind and does not deviate. This is a simple
illustration of the influence of the vision's cone.

\begin{figure}[ht!]
\begin{center}
 \begin{tabular}{cc}
\includegraphics[width=8.cm]{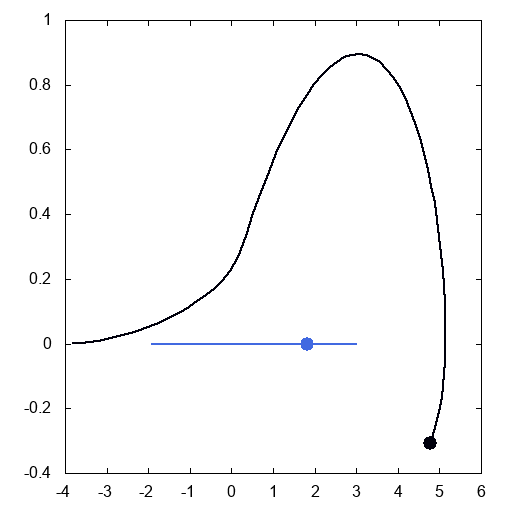} &    
\includegraphics[width=8.cm]{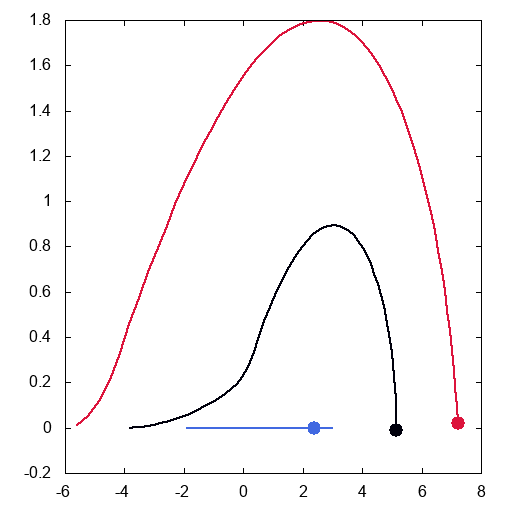} 
\\
(a)  & (b)  
\end{tabular}
\caption{\label{fig:test1-1}
{\bf Influence of the vision's cone.} space trajectory in
the horizontal plane for (a) 2 particles and  (b) 3 particles.}
 \end{center}
\end{figure}

\subsection{Collision avoidance in 3D}
We then consider the situation where all particles move in a three
dimensional space. All the particles are initially located in a ball and want to
move on the opposite direction with respect to the center of the ball. Therefore in this situation, 
the collision point of all particles is the center of the ball. 

We consider the microscopic model \eqref{eq:dynamics} without any
noise $\nu=0$ and choose $R=1$, and for the vision cone given in
Definition \ref{VisionCone} we take $\kappa=\cos(2\pi/3)$ whereas the rotation axis and the  turning
frequency are given in  (\ref{model:0})-(\ref{model:2}).

In that case we recover a situation similar to the previous one
but in three dimensions and the rotation axis  is no more colinear
to the $\be_z$ unit vector. Thanks to the turning operator, the
collision is avoided  and the particles have a smooth trajectory in 3D
as it can be shown in Figure \ref{fig:test2-1} for two or three
particles. With more particles we recover the same kind of results as for the
motion in the horizontal plane.

\begin{figure}[ht!]
\begin{center}
 \begin{tabular}{cc}
\includegraphics[width=8.cm]{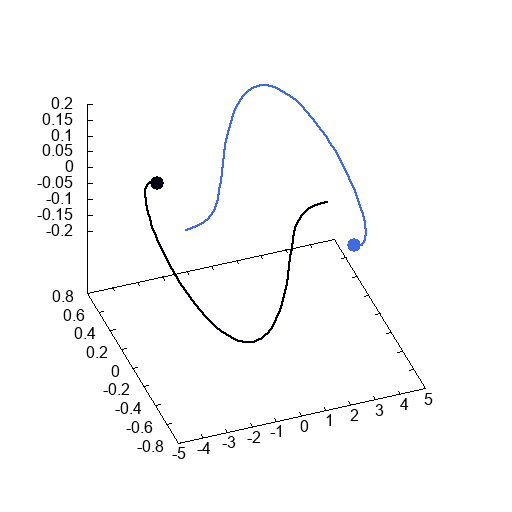} &    
\includegraphics[width=8.cm]{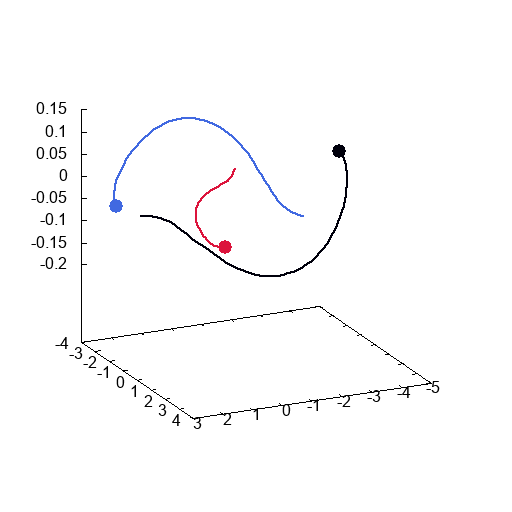} 
\\
(a)  & (b)  
\end{tabular}
\caption{\label{fig:test2-1}
{\bf Collision avoidance in 3D.} space trajectories in
three dimension for (a) 2 particles and  (b) 3 particles.}
 \end{center}
\end{figure}

\subsection{Moving around obstacles}
In this last example, we consider the motion of particles in presence
of fixed obstacles. The collision avoidance process follows the line
of Section \ref{sec:nc-int} with non-cooperative interactions. We
first introduce the point $O$ which represents the point of coordinate $x_O$ defined as 
$$
\bx_O \,=\, \argmin_{\bx \in \partial O \cap \mK_i(t^0)}
d\left(\bx_i(t),\bx\right)
$$
hence the force field acting on the particle $i\in\{1,\ldots,N\}$ is 
$$
\bF^{\rm obs}(\bx_i, \bv_i)  \,\,=\,\, \omega_{iO}\, H(\alpha_{iO})\, \bv_i \wedge \br_{iO}, 
$$
with $\omega_{i0}>0$ and  a rotation axis $\br_{iO}$ given by
\begin{equation*}
\br_{iO} \,:=\, \frac{\bv_i\wedge \bk_{iO}}{d_{iO}}\,,
\end{equation*}
with $\bk_{iO}$ the unit vector in the direction $\bx_O-\bx_i$ and
$d_{iO}=|\bx_O-\bx_i|$. The function $H_{iO}$ is given by (\ref{model:1bis}) and the
frequency $\omega_{iO}>0$ is 
$$
\omega_{iO} \,=\, \frac{16\pi}{|\br_{iO}|} \,e^{-\tau_{i0}}, 
$$
with $\tau_{iO}>0$ given by (\ref{TimeToInteraction}).

The particles are attracted to the target $\bx_T=(7,7,0)$, whereas the
obstacles are represented by two balls $B(\bx_0,1/2)$ and $B(\bx_1,1)$
with $\bx_0=(2,2,0)$ and $\bx_1=(5,5,0)$. 

We represent in Figure \ref{fig:test3-1} the space trajectories at
different time. The particles are initially located on a sphere
centered in $(-1,-1,0)$ with a random velocity. On the one hand we
observe that due to the attractive potential, all particles choose the
same direction and thanks to the collision avoidance operator, they do
not collide.  On the other hand, when they approach the obstacle they
deviate and remain relatively far from the obstacles. Finally at time
$t=20$, all particles are moving around the target point.
 
\begin{figure}[ht!]
\begin{center}
 \begin{tabular}{cc}
\includegraphics[width=8.cm]{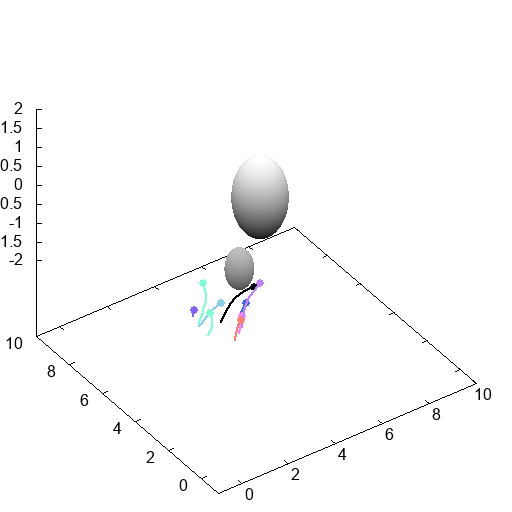} &    
\includegraphics[width=8.cm]{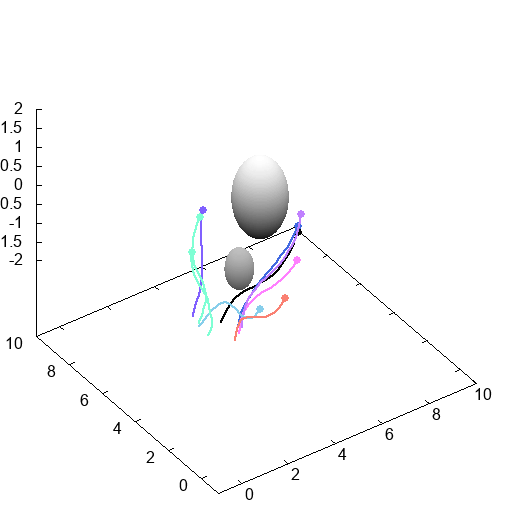} 
\\
$t=05$  & $t=10$
\\
\includegraphics[width=8.cm]{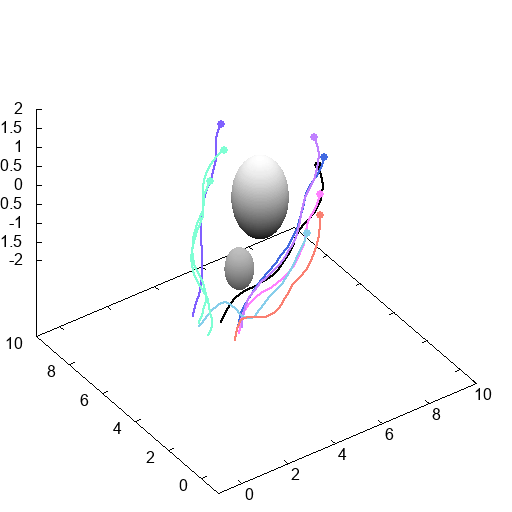} &    
\includegraphics[width=8.cm]{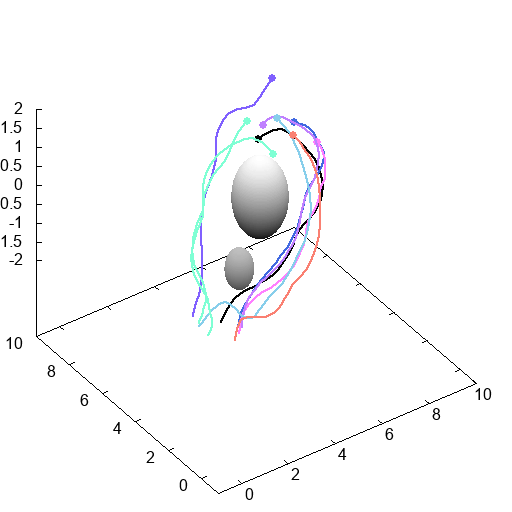} 
\\
$t=15$  & $t=20$
\end{tabular}
\caption{\label{fig:test3-1}
{\bf Moving around obstacles.} space trajectory in
three dimension at different time $t=5, \, 10, \,15$ and $20$.}
 \end{center}
\end{figure}

\section{Conclusion and Perspectives}
\setcounter{equation}{0}
\label{sec:5}
In this article, we have proposed  a three dimensional dynamical model
for collision avoidance based on previous works in two dimension for
pedestrian flows \cite{Degond-1, Degond-2,Moussaid}. This individual
based model relies on a vision-based framework: the particles analyze the scene
and react to the collision threatening partners by changing their
direction of motion. We have also  proposed a kinetic version of this
individual based model and perform some numerical experiments which
illustrate the ability of the microscopic model to avoid collisions in
three dimensions.

In a future work, the approach  developed in Section
\ref{sec:3}, which is based on a mean field model, will be investigated to
study the collision avoidance process in the presence of many
vehicles. Indeed for a large number of particles, sensors are not
able to distinguish each individual but only clouds of particles are
detected, the application of mean field models may contribute on the
design of efficient algorithms since the sum of interacting particles
is replaced by a self consistent force. 

On the other hand, more precise models can be applied to describe the
motion in three dimension of vehicles as multi-agent dynamics where each agent is described
by its position and body attitude. More precisely, each agent travels in a given direction
and its frame can rotate around it adopting different configurations. In this
manner, the frame attitude is described by three orthonormal axes
giving rotation matrices \cite{DFM}.

\section{Aknowledgement}
The authors thank anonymous referees and highly appreciate their valuable
comments and suggestions, which significantly contributed to improve
the quality of the justification of the model in Section \ref{sec:2.2}.  
\bibliographystyle{plain}

\begin{flushleft} \signCP \end{flushleft}
\begin{flushright} \signFF \end{flushright}

\end{document}